\documentclass[12pt,a4paper,leqno]{amsbook}
\usepackage{graphicx}
\newcounter{minutes}
\divide\time by 60
\newcounter{hours}
\multiply\time by 60
\addtocounter{minutes}{-\time}
\newcommand{\R}{\mathbb{R}}
\newcommand{\Rn}{\mathbb{R}^n}
\newcommand{\p}{\varphi}
\newcommand{\e}{\varepsilon}
\newcommand{\diam}{\textnormal{diam} \,}
\newcommand{\im}{\textnormal{Im}\,}
\font\fFt=eusm10 
\font\fFa=eusm7  
\font\fFp=eusm5  
\def\K{\mathchoice
{\hbox{\,\fFt K}}
{\hbox{\,\fFt K}}
{\hbox{\,\fFa K}}
{\hbox{\,\fFp K}}}
\usepackage{amssymb}
\usepackage{graphicx}
\usepackage[T1]{fontenc}
\usepackage[english]{babel}
\usepackage{multirow}
\date{}
\newfont{\cyrilic}{wncyr10 scaled 1000}
\title{MODULI OF CONTINUITY OF QUASIREGULAR MAPPINGS}
\author{Vesna Manojlovi\' c \\
\textrm{email:} {\tt vesnak@fon.bg.ac.yu}\\
{\tiny \tt File:~\jobname .tex,
          printed: \number\year-\number\month-\number\day,
          \thehours.\ifnum\theminutes<10{0}\fi\theminutes}
 }
\newcommand{\comment}[1]{}

\swapnumbers
\theoremstyle{plain}

\newtheorem{theorem}[equation]{Theorem}
\newtheorem{lemma}[equation]{Lemma}
\newtheorem{definition}[equation]{Definition}
\newtheorem{example}[equation]{Example}
\newtheorem{corollary}[equation]{Corollary}

\newtheorem{remark}[equation]{Remark}
\newtheorem{problem}[equation]{Problem}
\newtheorem{question}[equation]{Open question}

\newcommand{\pa}{\partial}

\numberwithin{equation}{section}

\newenvironment{subsec}
{
 \addtocounter{equation}{1}
 \bigskip
 \noindent
   {\bf \arabic{section}.\arabic{equation}.}
\begin{textbf}
} {\end{textbf}}

\begin{document}
\begin{titlepage}
\begin{center}
\bf FACULTY OF MATHEMATICS\\
UNIVERSITY OF BELGRADE
\vspace{5em}

\LARGE
DOCTORAL DISSERTATION
\vspace{5em}

{
\bf MODULI OF CONTINUITY OF QUASIREGULAR MAPPINGS
\rule{0ex}{1.3em}
}
\end{center}
\vspace{18em}

\noindent
\bf CANDIDATE\hfill MENTOR \hfill COMENTOR
\vspace{1em}

\noindent
\small
VESNA MANOJLOVI\' C \hfill MIODRAG MATELJEVI\' C \hfill MATTI VUORINEN
\vspace{3em}

\begin{center}
BELGRADE 2008
\end{center}
\end{titlepage}

\tableofcontents

\pagebreak

\thispagestyle{empty}
\begin{center}
\Large
Acknowledgements
\end{center}
\vspace{3em}

I would like to express my gratitude to Prof. Arsenovi\' c, Prof. Mateljevi\' c and Prof. Pavlovi\' c
whose lessons and lectures seminars at graduate studies at Faculty of Mathematics, University of Belgrade
gave me a solid knowledge base for my further research.

It is at one of these seminars in December of 2006 that I met Prof. Matti Vuorinen from the University of Turku,
Finland at the time when I was looking for a topic of Master Thesis.  Our cooperation led to my Master Thesis
in June of 2007, my three visits to Turku, in the period from December 2007 to June 2008, to do research in
beautiful area of Quasiconformal mappings, which resulted in our two joint papers and this Dissertation.
It is difficult to overstate my appreciation for his invaluable and crucial guidance throughout this process.
Most of the first chapter is an elaboration of his ideas on natural metrics.

I wish to thank Prof. Arsenovi\' c for his continuous support and advice given in relation to harmonic functions.
Also, many thanks to Prof. Mateljevi\' c for useful exchange of ideas related to harmonic maps.

Special thanks to Prof. Pavlovi\' c whose contribution is strongly felt in the second chapter of this thesis.
\vspace{5em}

Belgrade, July 2008
\vspace{1em}

{\it Vesna Manojlovi\' c}
\vfill

\pagebreak

\thispagestyle{empty}
\begin{center}
\Large
Summary
\end{center}
\vspace{3em}

This thesis consists of Chapters 1 and 2. The main results are
contained in the two preprints and two published papers, listed below.

Chapter 1 deals with conformal invariants in the euclidean space
$\mathbb{R}^n, n\ge2,$ and their interrelation. In particular,
conformally invariant metrics and balls of the respective metric
spaces are studied. Another theme in Chapter 1 is the study of
quasiconformal maps with identity boundary values in two different
cases, the unit ball and the whole space minus two points. These
results are based on the two preprints:

{\sc R. Kl\'en, V. Manojlovi\'c and M. Vuorinen:}  \emph{
Distortion of two point normalized quasiconformal mappings,}
arXiv:0808.1219[math.CV], 13 pp.,

{\sc V. Manojlovi\'c and M. Vuorinen:}  \emph{
On quasiconformal maps with identity boundary values,}
arXiv:0807.4418[math.CV], 16 pp.

Chapter 2 deals with harmonic quasiregular maps.
Topics studied are: Preservation of modulus of continuity, in particular
Lipschitz continuity, from the boundary to the interior of domain in case of
harmonic quasiregular maps and quasiisometry property of harmonic quasiconformal maps.
Chapter 2 is based mainly on the two published papers:

{\sc M. Arsenovi\'c, V. Koji\'c and M. Mateljevi\'c:}
{\em On Lipschitz continuity of harmonic quasiregular maps on the
unit ball in $\bold R\sp n$.},
Ann. Acad. Sci. Fenn. Math. 33 (2008), no. 1, 315--318.

{\sc V. Koji\' c and M. Pavlovi\' c:}
{\it Subharmonicity of $|f|^p$ for quasiregular harmonic functions, with applications},
J. Math. Anal. Appl. 342 (2008) 742-746

\vfill

\pagebreak

\chapter{Quasiconformal Mappings}
{
\section{Introduction}
}

Conformal invariance has played a predominant role in the study of
geometric function theory during the past century. Some of the
landmarks are the pioneering contributions of Gr\"otzsch and
Teichm\"uller prior to the Second World War, and the paper of
Ahlfors and Beurling \cite{ahb} in 1950. These results lead to
farreaching applications and have stimulated many later studies
\cite{k}. For instance, Gehring and V\"ais\"al\"a \cite{g3},
\cite{v1} have built the theory of quasiconformal mappings in
$\mathbb{R}^n$ based on the notion of the modulus of a curve family
introduced in \cite{ahb}.

In the first chapter of this dissertation our goal is to study two kinds
of conformally invariant
extremal problems, which in special cases reduce to problems due to
Gr\"otzsch and Teichm\"uller, resp. These two classical extremal
problems are extremal problems for moduli of ring domains. The
Gr\"otzsch and Teichm\" uller rings are the extremal rings for
extremal problems of the following type, which were first posed for
the case of the plane. Among all ring domains which separate two
given closed sets $E_1$ and $E_2$, $E_1\cap E_2=\emptyset$, find one
whose module has the greatest value.

In the general case these extremal problems lead to conformal
invariants $\lambda_G(x,y)$ and $\mu_G(x,y)$ defined for a domain $G
\subset \mathbb{R}^n$ and $x,y \in G \,.$ A basic fact is that
$\lambda_G(x,y)^{1/(1-n)}$ and $\mu_G(x,y)$ are metrics. Following
closely the ideas developed in \cite{vu1} and \cite{vu2} we study
three topics: (a) the geometry of the metric spaces $(G,d)$ when $d$
is $\lambda_G(x,y)^{1/(1-n)}$ or $\mu_G(x,y)$, (b) the relations of
these two metrics to several other metrics and (c) the behavior of
quasiconformal mappings with respect to several of these metrics.
One of our main results is to present a revised version of the Chart
on p. 86 of \cite{vu1}, taking into account some later developments,
such as \cite{h}, \cite{hv}, \cite{vu2}.

Then we present an application to the geometry of balls in these
metrics. As a special case we investigate $\lambda$ metric
in $B^2\setminus\{0\}$, continuing work of \cite{h}.

Another question we address is: if
$f:(G_i,m_{G_i})\longrightarrow(G_i',m_{G_i'})$ is uniformly
continuous ($i=1,2$), is $(G,m_{G})\longrightarrow(G',m_{G'})$
uniformly  continuous ($G=G_1\cup G_2$, $G'=G_1'\cup G_2'$)?

Chapter 1 concludes with displacement estimates for $K$-qc mappings
which are identity on the boundary of $G$.

In the second chapter we explore what additional information on a
$K$-qc mappings we get if we assume it is also harmonic. We call such mappings
hqc-mappings.

In case $n=2$ we show that hqc map has the same type of moduli of continuity
on $\overline D$ as on $\partial D$.

A similar, for the Lipschitz case, result is proved on $B^n$.
Finally, we show, for $n=2$, that any hqc map is bilipschitz in
quasihyperbolic metric.

\bigskip

\section{The extremal problems of Gr\"otzsch and Teichm\"uller}

In what follows, we adopt the standard definitions notions related
of quasiconformal mappings from \cite{v1}.

We use notation $B^n(x,r) = \{ y \in \mathbb R^n \colon |x-y| < r \}$,
 $S^{n-1}(x,r) = \{ y \in \mathbb R^n \colon |x-y| = r \}$,
  $H^n = \{ (x_1, \dots ,x_n) \in \mathbb R^n \colon x_n > 0 \}$
  and abbreviations $B^n(r) = B^n(0,r)$, $B^n = B^n(1)$,
  $S^{n-1}(r) = S^{n-1}(0,r)$ and $S^{n-1} = S^{n-1}(1)$.

For the modulus $M(\Gamma)$ of a curve family $\Gamma$ and its basic properties we refer the reader to \cite{v1}. Its basic property is
conformal invariance.

For $E,F,G\subset\overline{\mathbb R}^n$ let $\Delta(E,F,G)$ be the
family of all closed curves joining $E$ to $F$ within $G$. More
precisely, a path $\gamma:[a,b]\rightarrow\overline{\mathbb R}^n$
belongs to $\Delta(E,F,G)$ iff $\gamma(a)\in E$, $\gamma(b)\in F$
and $\gamma(t)\in G$ for $a<t<b$.

If $G$ is a proper subdomain of $\overline{\mathbb R^n}$, then for
$x,y\in G$ with $x\neq y$ we define
\begin{equation}
\lambda_G(x,y)=\inf_{C_x,C_y}M(\Delta(C_x,C_y;G))
\end{equation}
where $C_z=\gamma_z[0,1)$ and $\gamma_z:[0,1)\longrightarrow G$
is a curve such that $\gamma_z(0)=z$ and $\gamma_z(t)\rightarrow\partial G$
when $t\rightarrow 1$, $z=x,y$. This conformal invariant was introduced
by J. Ferrand (see \cite{vu2}).

For $x\in\mathbb{R}^n\setminus\{0,e_1\}$, $n\geqslant 2$, define
\begin{equation} \label{pxdef}
p(x)=\inf_{E,F}M(\Delta(E,F)),
\end{equation}
where the infimum is taken over all pairs of continua $E$ and $F$ in
$\overline{\mathbb R^n}$ with $0,e_1\in E$, $x,\infty\in F$.
This extremal quantity was introduced
by O. Teichm\"uller (see \cite{vu2}, \cite{hv}).

For a proper subdomain $G$ of $\overline{\mathbb{R}^n}$ and for all
$x,y\in G$ define
\begin{equation}
\mu_G(x,y)=\inf_{C_{xy}}M(\Delta(C_{xy},\partial G;G))
\end{equation}
where the infimum is taken over all continua $C_{xy}$ such that
$C_{xy}=\gamma[0,1]$ and $\gamma$ is a curve with $\gamma(0)=x$
and $\gamma(1)=y$. For the case $G=B^n$ the function $\mu_{B^n}(x,y)$
is the extremal quantity of H. Gr\"otzsch (see \cite{vu2}).

Let $(X,d_1)$ and $(Y,d_2)$ be metric spaces and let $f:X \to Y$ be
a continuous mapping. Then we say that $f$ is uniformly continuous
if there exists an increasing continuous function $\omega:[0,\infty)
\to [0,\infty)$ with $\omega(0) =0$ and $d_2(f(x), f(y)) \le
\omega(d_1(x,y))$ for all $x,y \in X\,.$ We call the function
$\omega$ the modulus of continuity of $f\,.$ If there exist $C,
\alpha >0$ such that $\omega(t) \le C t^{\alpha}$ for all $t>0\,,$
we say that $f$ is H\"older-continuous with H\"older exponent
$\alpha \,.$ If $\alpha=1\,,$ we say that $f$ is Lipschitz with the
Lipschitz constant $C$ or simply $C$-Lipschitz. If  $f$ is a
homeomorphism and both $f$ and $f^{-1}$ are $C$-Lipschitz, then $f$
is $C$-bilipschitz or $C$-quasiisometry and if $C=1$ we say that $f$
is an isometry. These conditions are said to hold locally, if they
hold for each compact subset of $X\,.$

A very special case of these are isometries.

Let $(X_1,d_1)$ and $(X_2,d_2)$ be metric spaces and let
$f:X_1\rightarrow X_2$ be a homeomorphism. We call $f$ an {\it
isometry} if $d_2(f(x),f(y))=d_1(x,y)$ for all $x,y\in X_1$.

In this section we introduce five types of metrics:
\begin{enumerate}
\item Spherical (chordal) metric $q$.
\item
Quasihyperbolic metric $k_G$ of a domain $G \subset \mathbb R^n$.
\item
A metric $j_G$ closely related to $k_G$.
\item
Seittenranta's metric $\delta_G$.
\item
Apollonian metric $\alpha_G$.
\end{enumerate}
The first one is defined on $\overline{\mathbb R^n}=\mathbb{R}^n\cup\{\infty\}$.
The second and the third ones are
defined in any proper subdomain $G\subset\mathbb R^n$, both of them generalize
hyperbolic metric (on $B^n$ or $H^n$) to arbitrary proper subdomain
$G\subset\mathbb R^n$. Seittenranta's metric is natural, M\" obius invariant
analogue of the $j_G$-metric. Apollonian metric is defined in any proper subdomain
$G\subseteq\mathbb R^n$ which boundary is not a subset of a circle or a line.

\begin{subsec}{The spherical metric.}
{\rm The metric $q$ is defined by
\begin{equation}
q(x,y)=
\left\{
\begin{array}{ll}
\displaystyle
\frac{|x-y|}{\sqrt{1+|x|^2}\sqrt{1+|y|^2}}, & x\neq\infty\neq y,\vspace{0.5em}\\
\displaystyle
\frac{1}{\sqrt{1+|x|^2}}, & y=\infty.
\end{array}
\right.
\end{equation}
}
\end{subsec}

Absolute (cross) ratio of an ordered quadruple $a,b,c,d$ of distinct points
in $\overline{\mathbb R^n}$ is defined
\begin{eqnarray}
\label{quadruple}
|a,b,c,d|=\frac{q(a,c)\,q(b,d)}{q(a,b)\,q(c,d)}.
\end{eqnarray}

Now we introduce distance ratio metric or {\it $j_G$-metric}. For an open set
$G\subset\mathbb{R}^n$, $G\neq\mathbb{R}^n$ we define
$d(z)=d(z,\partial G)$ for $z\in G$ and
\begin{equation}
j_G(x,y)=\log
\left(
1+\frac{|x-y|}{\min\{d(x),d(y)\}}
\right)
\end{equation}
for $x,y\in G$.

For a nonempty $A\subset G$ we define the $j_G$-diameter of $A$ by
$$
j_G(A)=\sup\{j_G(x,y)\,|\,x,y\in A\}.
$$

For an open set $G\subset\mathbb{R}^n$, $G\neq\mathbb{R}^n$, and a nonempty
$A\subset G$ such that $d(A,\partial G)>0$ we define
$$
r_G(A)=\frac{d(A)}{d(A,\partial G)}.
$$

If $\rho(x)>0$ for $x\in G$, $\rho$ is continuous and if $\gamma$ is a rectifiable curve
in $G$, then we define
$$
l_{\rho}(\gamma)=\int_\gamma\rho\,ds.
$$
The Euclidean length of a curve $\gamma$ is denoted by $l(\gamma)$.

Also, for $x_1,x_2\in G$ we define
\begin{eqnarray}
d_\rho(x,y)=\inf l_\rho(\gamma),
\end{eqnarray}
where the infimum is taken over all rectifiable curves from $x_1$ to $x_2$.

It is easy to show that $d_\rho$ is a metric in $G$.

Now we take any proper domain $G\subset\mathbb R^n$ and set
$\rho(x)=\frac 1{d(x,\partial G)}$.

The corresponding metric, denoted by $k_G$, is called the
{\it quasihyperbolic metric} in
$G$. Since,
$$
\rho(\varphi(x))=\frac{1}{d(\varphi(x),\partial\, (\varphi G))}=
\frac{1}{d(x,\partial G)}=\rho(x),
$$
for Euclidean isometry $\varphi$,
$$
k_{G'}(x',y')=k_G(x,y),\qquad\mbox{where }G'=\varphi(G),\,\,\,x'=\varphi(x),
\,\,\,y'=\varphi(y).
$$

Now we introduce {\it Seittenranta's metric $\delta_G$} \cite{se}.
For more details on M\" obius transformations in $\mathbb R^n$ see \cite{be}.
For an open set $G\subset\mathbb{R}^{n}$
with ${\rm card}\partial G\geqslant 2$ we set
$$
m_{G}(x,y)=\sup_{a,b\in\partial G}|a,x,b,y|
$$
and
$$
\delta_{G}(x,y)=\log(1+m_{G}(x,y))
$$
for all $x,y\in G$.

Consider now the case of an unbounded domain $G \subset {\mathbb R}^n, \infty \in \partial G \,.$
Note that if $a$ or $b$ in the supremum equals infinity, then we get exactly $j_G$
metric. This implies that we always have $j_G\leqslant\delta_G$.

We will also use {\it Apollonian metric} studied by Beardon \cite{be2},
(also see \cite[7.28 (2)]{avv}) defined in open proper subsets $G\subset\mathbb{R}^n$ by
$$
\alpha_{G}(x,y)=\sup_{a,b\in\partial G}\log|a,x,y,b|\quad\mbox{for all }x,y\in G.
$$
This formula defines a metric iff $\mathbb{R}^n\setminus G$ is not contained in
an $(n-1)$-dimensional sphere in $\mathbb{R}^n$.

In general, the hyperbolic-type metrics can be divided into length-metrics,
defined by means of integrating a weight function and point-distance metric.

Another group may again be classified by the number of boundary
points used in there's definition. So for instance, the $j$ metric is
one-point metric, while the Apollonian metric is two-point metric.
\begin{definition}
\rm
A domain $A\subset\overline{\mathbb R^n}$ is a ring if $C(A)$ has exactly two
components, where $C(A)$ denotes the complement of $A\subset\mathbb{R}^n$.
\end{definition}

If the components of $C(A)$ are $C_0$ and $C_1$, we denote
$A=R(C_0,C_1)$, $B_0=C_0\cap\overline A$ and $B_1=C_1\cap\overline
A$. To each ring $A=R(C_0,C_1)$, we associate the curve family
$\Gamma_A=\Delta(B_0,B_1,A)$ and the modulus of $A$ is defined by
$\mod(A)=M(\Gamma_A)$. Next, the capacity of $A$ is by definition
${\rm cap} A =\omega_{n-1} ( \mod A)^{1-n}$.

The complementary components of the Gr\" otzsch ring $R_{G,n}(s)$ in $\mathbb
R^n$ are $\overline{B}^n$ and $[s\cdot e_1,\infty]$, $s>1$, while those of the
Teichm\" uller ring $R_{T,n}(t)$ are $[-e_1,0]$ and $[t\,e_1,\infty]$, $t>0$.
We shall need two special functions $\gamma_n(s)$, $s>1$, and $\tau_n(t)$,
$t>0$, to designate the moduli of the families of all those curves which connect
the complementary components of the Gr\"otzsch and Teichm\" uller rings in
$\mathbb R^n$, respectively.
\begin{eqnarray*}
\gamma_n(s)=M(\Gamma_s)=\gamma(s),\quad\Gamma_s=\Gamma_{R_{G,n}}(s),\\
\tau_n(t)=M(\Delta_t)=\tau(t),\quad\Delta_t=\Gamma_{R_{T,n}}(t).
\end{eqnarray*}

These functions are related by a functional identity \cite[Lemma 6]{ge}
\begin{equation}
\label{gt}
\gamma_n(s)=2^{n-1}\tau_n(s^2-1).
\end{equation}

\begin{definition}
\rm
Given $r>0$, we let $R\Psi_n(r)$ be the set of all rings $A=R(C_0,C_1)$ in
$\overline{\mathbb R^n}$ with the following properties:
\begin{enumerate}
\item
$C_0$ contains the origin and a point $a$ such that $|a|=1.$
\item
$C_1$ contains $\infty$ and a point $b$ such that $|b|=r$.
\end{enumerate}
Teichm\" uller first considered the following quantity in the planar
case ($n=2$):
$$
\tau_n(r)=\inf M(\Gamma_A)=\inf\{p(x)\,|\,|x|=r\},
$$
where the infimum is taken over all rings $A\in R\Psi_n(r)$ and $p(x)$ is
as in (\ref{pxdef}). For
$n\geqslant 3$ it was studied in \cite{ge} and in \cite{hv}.
\end{definition}
\begin{theorem}
\cite[Theorem 11.7]{v1} \label{H_n_properties}  The function
$\tau_n:(0,\infty)\rightarrow\mathbb (0,\infty)$ has the following
properties:
\begin{enumerate}
\item
$\tau_n$ is decreasing,
\item
$\lim_{r\rightarrow\infty}\tau_n(r)=0\, ,$
\item
$\lim_{r\rightarrow 0}\tau_n(r)=\infty\, ,$
\item
$\tau_n(r)>0$ for every $r>0$.
\end{enumerate}
Moreover, $\tau_n:(0,\infty)\rightarrow\mathbb (0,\infty)$ and
$\gamma_n:(1,\infty)\rightarrow\mathbb (0,\infty)$ are
homeomorphisms.
\end{theorem}

From the definition of $\tau_n$ and from the conformal invariance of
the modulus, we obtain the following estimate:
\begin{theorem}
\label{function_chi_n}
Suppose that $A=R(C_0,C_1)$ is a ring and
that $a,b\in C_0$ and $c,\infty\in C_1$. Then
$$
M(\Gamma_A)\geqslant\tau_n\left(\frac{|c-a|}{|b-a|}\right).
$$
Here equality holds for the Teichm\"uller ring, when $a=0, b= -e_1,
c= te_1, t>0$ and $C_0=[-e_1,0], C_1=[te_1,\infty).$
\end{theorem}

\begin{theorem}
\label{extremality} Let $C\subset B^n$ be a connected compact set
containing $0$ and $x$, where $|x|<1$. Then the capacity of a ring
domain with components $C_0=C$, $C_1=\{x:|x|\geqslant 1\}$ is at
least $\gamma_n(\frac 1{|x|})$. Here equality holds for the ring
with the complementary components $[0,|x|e_1]$ and ${\mathbb R}^n
\setminus B^n .$
\end{theorem}
These theorems state the extremal properties of the Teichm\"uller
and Gr\" otzsch rings and their proofs are based on the
symmetrization theorem in \cite[Theorem 1]{ge}.

\section{Moduli of continuity}

In this section we investigate the moduli of continuity of the identity
mappings $id_G:(G,\rho)\longrightarrow(G,d)$ where $\rho$ and $d$ are chosen from the
set of interesting metrics defined on $G$ (like quasihyperbolic metric $k$, modulus metric
$\mu$ etc.).

Hence, we are interested in results of type
\begin{equation}
\label{interesting}
d(x,y)\leqslant\zeta(\rho(x,y))=\zeta_\rho^d(\rho(x,y)),\quad x,y\in G.
\end{equation}
We give several estimates of this type, and then we collect these
results in a charts at the end of this section.

Note that in our charts we have $\lambda_G^{-1}$, as well as in the
inequalities of type (\ref{interesting}); however reader should be
aware that in general $\lambda_G^{-1}$ is not a metric. In fact
$\lambda_G^{1/{1-n}}$ is always a metric. For more details on
this matter see \cite{vu4}.

It is well known that $j_G(x,y)\leqslant k_G(x,y)$, so $\zeta_k^j(t)=t$.

\begin{lemma}
For $x,y\in G$
$$
k_G(x,y)\geqslant\log\left(1+\frac{m(x,y)}{\min\{d(x),d(y)\}}\right)\geqslant j_G(x,y).
$$
where $m(x,y)=\inf\{l(\gamma)\,|\,\gamma\mbox{ is a curve joining }x\mbox{ and }y\mbox{ in }G\}$.
\end{lemma}

\begin{proof}
We may assume $0<d(x)\leqslant d(y)$. Choose a rectifiable arc
$\gamma:[0,s]\rightarrow G$ from $x$ to $y$, parametrized by arc
length:
$$
\gamma(0)=x,\qquad
\gamma(s)=y;
$$
obviously $s\geqslant|x-y|$. For any $0\leqslant t\leqslant s$ we have
$$
d(\gamma(t))\leqslant d(x)+t,\qquad\mbox{(a key observation)},
$$
so,
$$
l_\rho(\gamma)\geqslant\int_0^s\frac{dt}{d(x)+t}=\log\frac{d(x)+s}{d(x)}\geqslant
\log\frac{d(x)+|x-y|}{d(x)}=j_G(x,y).
$$
\end{proof}

The reverse inequality is not true in general; domain $G$ such that there is a constant
$c>0$ such that $k_G\leqslant c\,j_G$ is called uniform domain, so in that case
$\zeta_j^k(t)=ct$.
\begin{lemma}
\cite[Lemma 2.21]{vu1}
Let $G$ be a proper subdomain of $\mathbb R^n$. If $x\in G$, $d(x)=d(x,\partial G)$
and $y\in B^n(x,d(x))=B_x$, $x\neq y$, then
\begin{equation}
\lambda_G(x,y)\geqslant\lambda_{B_x}(x,y)\geqslant
c_n\log\left(\frac{d(x)}{|x-y|}\right)
\end{equation}
where $c_n$ is the positive number in \cite[(10.11)]{v1}. There
exists a strictly increasing function
$h_1:(0,+\infty)\longrightarrow(0,+\infty)$ with $\lim_{t\rightarrow
0_+}h_1(t)=0$ and $\lim_{t\rightarrow+\infty}h_1(t)=+\infty$,
depending only on $n$, such that
\begin{equation}
\lambda_G(x,y)\leqslant h_1\left(\frac{\min\{d(x),d(y)\}}{|x-y|}\right)
\end{equation}
for $x,y\in G$, $x\neq y$. If $x\in G$ and $y\in B^n(x,d(x))=B_x$, $x\neq y$, then
\begin{equation}
\label{mu-G}
\mu_G(x,y)\leqslant\mu_{B_x}(x,y)=capR_G\left(\frac{d(x)}{|x-y|}\right)\leqslant
\omega_{n-1}\left(\log\left(\frac{d(x)}{|x-y|}\right)\right)^{1-n}.
\end{equation}
\end{lemma}

From (\ref{mu-G}) we get $\mu_G(x,y)\leqslant\gamma\left(\frac{d(x)}{|x-y|}\right)$ for
$x\in G$ and $y\in B_x$. It is equivalent with
$\mu_G(x,y)\leqslant\gamma\left(\frac 1r\right)$ where $r=\frac{|x-y|}{d(x)}$.

We can express $j_G(x,y)$ in terms of $r$: $r=e^j-1$ and obtain
$$
\mu_G(x,y)\leqslant\gamma\left(\frac{1}{e^j-1}\right).
$$

This gives $\zeta_j^\mu(t)=\gamma\left(\frac 1{e^t-1}\right)$ locally.

\begin{lemma} \cite[Lemma 2.39]{vu1}
\label{vu1l2-39}
For $n\geqslant 2$ there exists strictly
increasing function $h_2:[0,+\infty)\longrightarrow[0,+\infty)$ with
$h_2(0)=0$ and $\lim_{t\rightarrow+\infty}h_2(t)=+\infty$ with the
following properties.

If $E$ is closed and $F$ is compact in $\mathbb R^n$ then
\begin{equation}
M(\Delta(E,F))\leqslant h_2(T);\quad T=\min\{j_{\mathbb R^n\setminus E}(F),j_{\mathbb R^n\setminus F}(E)\}.
\end{equation}
In particular, if $G$ is a proper subdomain of $\mathbb R^n$, then
\begin{equation}
\label{h2}
\mu_G(x,y)\leqslant h_2(3k_G(x,y))
\end{equation}
for all $x,y\in G$. Moreover, there are positive numbers $b_1,b_2$ depending only on $n$ such that
\begin{equation}
\mu_G(x,y)\leqslant b_1k_G(x,y)+b_2
\end{equation}
for all $x,y\in G$.
\end{lemma}

From (\ref{h2}) we have $\zeta_k^\mu(t)=h_2(3t)$.

\begin{lemma}
\cite[Lemma 2.44]{vu1} If $E,F\subseteq\mathbb R^n$ are disjoint
continua, then
$$
M(\Delta(E,F))\geqslant\bar{c}_n\min\{j_{\mathbb R^n\setminus E}(F),j_{\mathbb R^n\setminus F}(E)\}
$$
where $\bar c_n$ is a positive number depending only on $n$.
\end{lemma}

\begin{corollary}
\cite[Corollary 2.46]{vu1} If $E$ and $F$ are disjoint continua in
$\overline{\mathbb R^n}$ and $\infty\in F$, then
$$
M(\Delta(E,F))\geqslant c_nj_{\mathbb R^n\setminus F}(E).
$$
\end{corollary}

\begin{corollary}
\cite[Lemma 6.23]{vu4} Let $G\subseteq\mathbb R^n$ be a domain
$G\neq\mathbb R^n$ and connected boundary $\partial G$. Then
\begin{equation}
\label{mu-j}
\mu_G(a,b)\geqslant c_n j_G(a,b)
\end{equation}
holds for $a,b\in G$. If, in addition, $G$ is uniform, then
\begin{equation}
\label{mu-k}
\mu_G(a,b)\geqslant B\,k_G(a,b)
\end{equation}
for all $a,b\in G$.
\end{corollary}

The first part of this corollary gives $\zeta_\mu^j(t)=\frac 1{c_n}\,t$ if
$\partial G$ is connected. (\ref{mu-k}) gives $\zeta_\mu^k(t)=c\,t$
if $\partial G$ is connected and $G$ is uniform.

\begin{lemma}
\cite[Corollary 15.13]{avv}Let $G$ be a proper
subdomain of $\mathbb R^n$, $x$ and $y$ distinct points in $G$ and
$m(x,y)=\min\{d(x),d(y)\}$. Then
\begin{equation}
\label{lambda-tau}
\lambda_G(x,y)\leqslant
\sqrt{2}\tau\left(\frac{|x-y|}{m(x,y)}\right).
\end{equation}
\end{lemma}

From (\ref{lambda-tau}) using again $r=e^j-1$, $r=\frac{|x-y|}{m(x,y)}$, we have
$$
\sqrt{2}\tau(e^j-1)\geqslant\lambda_G,
$$
and then, since $\tau$ is decreasing, $e^j
 \le \tau^{-1}\left(\frac{\lambda_G}{\sqrt{2}}\right)$ and from here
$$
j\leqslant\log\left(1+\tau^{-1}\left(\frac 1{\sqrt{2}\,\lambda_G^{-1}}\right)\right).
$$
Finally we obtain
$\zeta_{\lambda^{-1}}^j(t)=\log\left(1+\tau^{-1}\left(\frac 1{\sqrt{2}\,t}\right)\right)$.

\begin{definition}
A closed set $E$ in $\mathbb R^n$ is called a $c$-quasiextremal distance set or $c$-QED
exceptional or $c$-QED set, $c\in(0,1]$, if for each pair of disjoint continua
$F_1,F_2\subseteq\overline{\mathbb R^n}\setminus E$
\begin{equation}
M(\Delta(F_1,F_2;\overline{\mathbb R^n}\setminus E))\geqslant cM(\Delta(F_1,F_2)).
\end{equation}
If $G$ is a domain in $\overline{\mathbb R^n}$ such that $\overline{\mathbb R^n}\setminus G$
is a $c$-QED set, then we call $G$ a $c$-QED domain.
\end{definition}

\begin{theorem}
\cite[Theorem 6.21]{vu3} Let $G$ be a $c$-QED domain in $\mathbb
R^n$. Then
\begin{equation}
\label{vu3t6-21}
\lambda_G(x,y)\geqslant c\tau(s^2+2s)\geqslant 2^{1-n}c\tau(s)
\end{equation}
where $s=\frac{|x-y|}{min(d(x),d(y))}$.
\end{theorem}

From the first inequality in (\ref{vu3t6-21}), taking into account $s=e^j-1$, we obtain
$$
\lambda^{-1}=\frac 1\lambda\leqslant\frac 1c\,\frac 1{\tau((s+1)^2-1)}=\frac 1c\,\frac 1{\tau(e^{2j}-1)}.
$$
This gives $\displaystyle\zeta_j^{\lambda^{-1}}(t)=\frac 1c\,\frac 1{\tau(e^{2t}-1)}$ for $c$-QED domain $G$.

Combining $\zeta_k^j$ and $\zeta_j^{\lambda^{-1}}$ we estimate $\lambda_G^{-1}$ in terms of $k_G$, so
$\zeta_k^{\lambda^{-1}}=\zeta_j^{\lambda^{-1}}\circ\zeta_k^j=\zeta_j^{\lambda^{-1}}$. In fact, we have
$$
\lambda_G^{-1}\leqslant\frac 1c\frac 1{\tau(e^{2j}-1)}\leqslant\frac 1{c\tau(e^{2k}-1)}.
$$

The fields $\zeta_\mu^{\lambda^{-1}}$, $\zeta_{\lambda^{-1}}^k$, $\zeta_{\lambda^{-1}}^\mu$
are obtained in the same fashion as
$\zeta_k^{\lambda^{-1}}$, namely as compositions of appropriate
functions $\zeta_\rho^d$. We use following inequalities.

For $\zeta_\mu^{\lambda^{-1}}$ we have
$$
\lambda_G^{-1}\leqslant\frac 1{c\tau(e^{2j}-1)}\leqslant\frac 1{c\tau(e^{2\mu/c_n}-1)}=
\frac 1{c\tau(e^{b\mu}-1)},
$$
where the second inequality follows from (\ref{mu-j}) and where $b=\frac{2}{c_n}$.

For $\zeta_{\lambda^{-1}}^k$ we have
$$
k_G\leqslant c\,j_G\leqslant c\log\left(1+\tau^{-1}\left(\frac 1{\sqrt 2\,\lambda_G^{-1}}\right)\right)
$$
and for $\zeta_{\lambda^{-1}}^\mu$ we have
$$
\mu_G\leqslant\gamma\left(\frac 1{e^j-1}\right)\leqslant
\gamma\left(\frac 1{e^{\log\left(1+\tau^{-1}\left(\frac 1{\sqrt 2\,\lambda_G^{-1}}\right)\right)}-1}\right)
=\gamma\left(\frac 1{\tau^{-1}\left(\frac 1{\sqrt 2\,\lambda_G^{-1}}\right)}\right).
$$
\begin{theorem}
\cite[Theorem 3.4]{se}
The inequalities $j_{G}\leqslant\delta_G\leqslant 2j_G$ hold for every open set $G\subset\mathbb{R}^n$.
\end{theorem}

So, we deduce that $\zeta_\delta^j(t)=t$ and $\zeta_j^\delta(t)=2t$.

\begin{theorem}
\cite[Theorem 4.2]{se}
Let $G\subset\mathbb{R}^n$ be a convex domain, then $j_G\leqslant\alpha_G$.
\end{theorem}

This means that $\zeta_\alpha^j(t)=t$ for convex domains.

\begin{theorem}
\cite[Theorem 6.2]{se} Let $G$ be a domain in $\mathbb{R}^n$, for
which ${\rm card}\,\partial G\geqslant 2$ and $\partial G$ is
connected. Then, for distinct points $x,y\in G$,
$$
\mu_G(x,y)\geqslant\tau_n\left(\frac{1}{e^{\delta_G(x,y)}-1}\right).
$$
\end{theorem}

Solving for $\mu$ and using a fact that $\tau_n$ is decreasing we get:
$$
\tau_n^{-1}(\mu_G(x,y))\leqslant\frac 1{e^{\delta_G(x,y)-1}}
$$
and from here
$$
\delta_G(x,y)\leqslant\log\left(1+\frac{1}{\tau_n^{-1}(\mu_G(x,y))}\right).
$$
Hence, $\zeta_\mu^\delta(t)=\log\left(1+\frac{1}{\tau_n^{-1}(t)}\right)$ if $\partial G$ is
connected and has at least two points.
\begin{theorem}
\cite[Theorem 6.5]{se}
Let $G\subset\overline{\mathbb{R}^n}$ be a domain with card $\partial G\geqslant 2$. Then
$$
\lambda_G(x,y)\leqslant\tau_n\left(\frac{m_G(x,y)}{2}\right).
$$
\end{theorem}

Expressing $\mu_G$ in terms of $\delta_G$ we get:
$$
\lambda_G(x,y)\leqslant\tau_n\left(\frac{e^{\delta_G}-1}2\right)
$$
and from here we obtain
$$
\delta_G(x,y)\leqslant
\log\left(1+2\tau_n^{-1}\left(\frac 1{\lambda_G^{-1}(x,y)}\right)\right).
$$
This means that $\zeta_{\lambda^{-1}}^\delta(t)=\log\left(1+2\tau_n^{-1}\left(\frac 1{t}\right)\right)$
for domains with $card(\partial G)\ge 2$.
\vspace{2em}

At first, we give a $4\times 4$ chart.
\vspace{1em}

\begin{tabular}{|c|p{3.7cm}|p{2.5cm}|p{3.3cm}|p{3.4cm}|}
    \hline
    & \hspace{7mm} $j_G$ & $k_G$ & $\mu_G$ & $\lambda^{-1}_G$ \\
    \hline
    \multirow{3}{1cm}{$j_G$} & 1 & 2 & 3 & 4 \\
    & $\zeta_j^j(t)=t$ & \begin{minipage}{2.5cm}$\zeta_j^k(t)=ct$\\ $G$ -- uniform\\ $\zeta_j^k(t)=\varphi(t)$\\ $G$ -- $\varphi$ domain\end{minipage} &
    \begin{minipage}{3.3cm}$\displaystyle\zeta_j^\mu(t)=\gamma\left(\frac 1{e^t-1}\right)$\\locally\end{minipage} &
    \begin{minipage}{3.5cm} $\displaystyle\zeta_j^{\lambda^{-1}}(t)=\frac 1{c\tau(e^{2t}-1)}$\\ $G$ -- $c$-QED domain \end{minipage} \\
    \hline
    \multirow{3}{1cm}{$k_G$} & 5 & 6 & 7 & 8 \\
    & \hspace{7mm} $\zeta_k^j(t)=t$  & $\zeta_k^k(t)=t$ & $\zeta_k^\mu(t)=h_2(3t)$ & $\zeta_k^{\lambda^{-1}}=\zeta_j^{\lambda^{-1}}$ \\
    \hline
    \multirow{3}{1cm}{$\mu_G$} & 9 & 10 & 11 & 12 \\
    & \begin{minipage}{2.5cm}$\displaystyle\zeta_\mu^j(t)=\frac 1{c_n}\cdot t$ \\ $\partial G$ connected\\[-0.5em]\rule{0mm}{0mm}\end{minipage} & \begin{minipage}{2.5cm}$\zeta_\mu^k(t)=c\cdot t$ \\ $G$ uniform\\[-0.5em]\rule{0mm}{0mm}$\partial G$ connected\end{minipage}
    & $\zeta_\mu^\mu(t)=t$ & \begin{minipage}{3.2cm}$\zeta_\mu^{\lambda^{-1}}=\zeta_\mu^j\circ\zeta_j^{\lambda^{-1}}$\\ $G$ -- $c$-QED domain\\ $\partial G$ connected\end{minipage} \\
    \hline
    \multirow{3}{1cm}{$\lambda^{-1}_G$} & 13 & 14 & 15 & 16 \\
    & $\zeta_{\lambda^{-1}}^j(t)=\log\left(1+\tau^{-1}\left(\frac 1{\sqrt{2}\,t}\right)\right)$ & \begin{minipage}{3cm}$\zeta_{\lambda^{-1}}^k=\zeta_{\lambda^{-1}}^j\circ\zeta_j^k$\\ $G$ uniform\\[-0.5em] \rule{0mm}{0mm}\end{minipage} & \begin{minipage}{3.4cm}$\zeta_{\lambda^{-1}}^\mu=\zeta_{\lambda^{-1}}^j\circ\zeta_j^\mu$\\locally\end{minipage}
    & $\zeta_{\lambda^{-1}}^{\lambda^{-1}}(t)=t$ \\
    \hline
  \end{tabular}
\vspace{2em}

Function $\zeta_j^\mu$ can be written in a different form using the
estimate of $\gamma$ function. We define functions $\Phi$ and $\Psi$
as in \cite[7.19]{vu2} by
\begin{equation}
\label{gamma-Phi}
\gamma_n(s)=\omega_{n-1}(\log(\Phi(s)))^{n-1},\quad s>1
\end{equation}
\begin{equation}
\label{Psi}
\tau_n(t)=\omega_{n-1}(\log(\Psi(t)))^{n-1},\quad t>0.
\end{equation}
\begin{lemma}
\cite[Lemma 7.22]{vu2} For each $n\geqslant 2$ there exists a number
$\lambda_n\in[4,2\,e^{n-1})$, $\lambda_2=4$, such that
\begin{equation}
t\leqslant\Phi(t)\leqslant\lambda_nt,\quad t>1
\end{equation}
\begin{equation}
t+1\leqslant\Psi(t)\leqslant\lambda_n^2(t+1),\quad t>0.
\end{equation}
\end{lemma}

From (\ref{Psi}) we have that $\omega_{n-1}(\log(\lambda_n^2(t+1)))^{1-n}\leqslant\tau_n(t)
\leqslant\omega_{n-1}(\log(t+1))^{1-n}$.

From (\ref{gamma-Phi}) we have
$$
\omega_{n-1}\left(\log\lambda_nt\right)^{1-n}\leqslant\gamma_n(t)
\leqslant\omega_{n-1}\left(\log t\right)^{1-n},\quad t>1.
$$
Using the right side of this inequality we have
$$
\gamma\left(\frac 1{e^t-1}\right)\leqslant\omega_{n-1}\left(\log\left(\frac 1{e^t-1}\right)\right)^{1-n}
\leqslant\omega_{n-1}\left(\log\left(\frac 1t\right)\right)^{1-n}.
$$
This gives $\zeta_j^\mu(t)\leqslant\omega_{n-1}\left(\log\left(\frac 1t\right)\right)^{1-n}$ locally.

\section{Inclusion relations for balls}

Each statement on modulus of continuity has its counterpart stated in terms of inclusions
of balls. Namely, if for some metrics $d_1$ and $d_2$ holds
$$
d_1(x,y)<t\Rightarrow d_2(x,y)<\zeta(t),
$$
then
$$
D_{d_1}(x,t)\subset D_{d_2}(x,\zeta(t)).
$$
A related question is to find, for a given $x\in G$ and $t>0$, minimal $\zeta(x,t)$
such that
$$
D_{d_1}(x,t)\subset D_{d_2}(x,\zeta(x,t)),
$$

This is circumscribed ball problem for a fixed $x\in G\,.$
\vspace{1em}

The quasihyperbolic ball $D_k(x,r)$ is the set $\{z\in
G\,|\,k_G(x,z)<r\}$, when $x\in G$ and $r>0$. By \cite[(3.9)]{vu2},
we have the inclusions
\begin{equation}
\label{quasiball}
B^n(x,r\,d(x))\subset D_k(x,M)\subset B^n(x,R\,d(x)),
\end{equation}
where $r=1-e^{-M}$ and $R=e^M-1$.

It was proved in \cite[15.13]{avv} that if $G$ is a proper subdomain of $\mathbb\mathbb R^n$
and if $x,y\in G$ with $x\neq y$, then
\begin{equation}
\label{lambda-tau-2}
\lambda_G(x,y)\leqslant\inf_{z\in\partial G}(\lambda_{\mathbb R^n\setminus\{z\}}(x,y))\leqslant\sqrt 2
\tau_n\left(\frac{|x-y|}{\min\{d(x),d(y)\}}\right)
\end{equation}
\begin{theorem}
\cite[Theorem 6.11]{h} \label{ht6-11} Let $G$ be a proper subdomain
of $\mathbb R^n$ and let $t>0$. We denote $c_1=\frac
1{(1+\tau_n^{-1}(t/\sqrt{2}))}$,
$c_2=\sqrt{\frac{\tau_n^{-1}(2t)}{(1+\tau_n^{-1}(2t))}}$ and
$c_3=\tau_n^{-1}(t/\sqrt 2)$, then the inclusions
\begin{equation}
D_{\lambda^{-1}}(a,t)\subset\{z\in G\,|\,d(z)>c_1 d(a)\},
\end{equation}
\begin{equation}
D_{\lambda^{-1}}(a,t)\supset B^n(a,c_2 d(a))\supset D_k(a,\log(c_2+1))
\end{equation}
and
\begin{equation}
\label{ht6-11-3}
D_{\lambda^{-1}}(a,t)\subset B^n(a,c_3 d(a))\cap G
\end{equation}
are valid for all $a\in G$. If, in addition, $t>\sqrt 2\tau_n(1)$, we have
that
\begin{equation}
\label{ht6-11-4}
B^n(a,c_3 d(a))\subset D_k(a,\log(1/(1-c_3))).
\end{equation}
\end{theorem}
To prove the inclusion (\ref{ht6-11-3}), we apply (\ref{lambda-tau-2}) to obtain
$$
\lambda_G(a,z)\leqslant \sqrt{2}\tau_n\left(\frac{|z-a|}{d(a)}\right).
$$
From here with the assumption $t\leqslant\lambda_G(a,z)$ we have $|z-a|<\tau_n^{-1}(t/\sqrt{2})d(a)$.

Since $D_{\lambda^{-1}}\subset G$, the inclusion (\ref{ht6-11-3}) holds.

Inclusion (\ref{ht6-11-4}) follows directly from (\ref{quasiball}) after we notice that the
condition $t>\sqrt{2}\tau_n(1)$ implies that $c_3<1$ and hence that the ball $B^n(a,c_3d(a))$
is included in $G$.
\begin{theorem}
\cite[Theorem 6.18]{h}
\label{ht6-18}
Let $G$ be a proper subdomain od $\mathbb R^n$ and assume that $G$ has a
connected, nondegenerate boundary. Let $t>0$ and denote
$d_1=\tau_n^{-1}(t)/(1+\tau^{-1}_n(t))$, $d_2=1/\gamma^{-1}_n(t)$
and $d_3=1/\tau^{-1}_n(t)$. Then, for all $a\in G$, the following inclusions hold
\begin{equation}
D_\mu(a,t)\subset\{z\in G\,|\,d(z)>d_1 d(a)\},
\end{equation}
\begin{equation}
\label{mu_e}
D_\mu(a,t)\supset B^n(a,d_2 d(a))\supset D_k(a,\log(d_2+1))
\end{equation}
\begin{equation}
D_\mu(a,t)\subset B^n(a,d_3 d(a))\cap G.
\end{equation}
If in addition $t<\tau_n(1)$, then
\begin{equation}
B^n(a,d_3 d(a))\subset D_k(a,\log(1/(1-d_3))).
\end{equation}
The numbers $d_1$, $d_2$ and $d_3$ are best possible for these inclusions.
\end{theorem}
We prove (\ref{mu_e}) only, because that part is used later on.

We assume that $a,z\in G$ and that $|z-a|\leqslant d_2\,d(a)$. Then, since
$\gamma_n^{-1}(t)>1$, we have $d(z,a)<d(a)$. We consider the following curve
families.
$$
\Gamma_J=\Delta(J_{az},\partial G;G),
$$
$$
\Gamma=\Delta(J_{az},S^{n-1}(a,d(a));\overline{B^n(a,d(a))}),
$$
and
\begin{equation}
\tilde\Gamma=\Delta([z',+\infty),S^{n-1};\mathbb{R}^n\setminus B^n),
\end{equation}
where $z'=\frac{d(a)}{|z-a|}\,e_1$. Since $J_{az}$ is a continuum
which joins $a$ and $z$, we have
\begin{equation}
\mu_G(a,z)\leqslant M(\Gamma_J)
\end{equation}
and since $\Gamma<\Gamma_J$, we have that $M(\Gamma_J)<M(\Gamma)$.

Using M\" obius transformations, we get
\begin{equation}
M(\Gamma)=M(\tilde\Gamma)=\gamma_n\left(\frac{d(a)}{|z-a|}\right),
\end{equation}
and since $|z-a|<d_2\,d(a)$ and $\gamma_n$ is a strictly decreasing
homeomorphism, it follows that
\begin{equation}
\gamma_n\left(\frac{d(a)}{|z-a|}\right)<\gamma_n\left(\frac 1{d_2}\right)=t.
\end{equation}

Combining all these inequalities, we get
$$
\mu_G(a,z)<t,
$$
which proves the left side of (\ref{mu_e}). The right side inclusion
follows from (\ref{quasiball}).

Theorem \ref{ht6-11} ((\ref{ht6-11-3}) and (\ref{ht6-11-4})) gives
\begin{theorem}
$$
\lambda^{-1}(a,b)<\frac 1t\Rightarrow k(a,b)<\log\frac 1{1-\tau_2^{-1}\big(\frac t{\sqrt 2}\big)},
\quad\mbox{for }t>\sqrt 2\tau_2(1)
$$
$$
\lambda^{-1}(a,b)<s\Rightarrow k(a,b)<\log\frac 1{1-\tau_2^{-1}\big(\frac 1{\sqrt 2s}\big)},
$$
$$
\zeta_{\lambda^{-1}}^k(s)=\log\frac 1{1-\tau_2^{-1}\big(\frac 1{\sqrt{2}s}\big)},\quad s<\frac 1{\sqrt 2\tau_2(1)}.
$$
Also we obtain $\lambda^{-1}(x,a)<\frac 1t\Rightarrow|x-a|<c_3d(a)<{\rm diam}(G)\,c_3(1/t)$ and from here
$\zeta_{\lambda^{-1}}^{|\cdot|}(t)=\tau_n^{-1}(1/(\sqrt 2t))\,{\rm diam}(G)\,.$
\end{theorem}

From Theorem \ref{ht6-18} we deduce
\begin{theorem}
In a domain $G$ with connected nondegenerate
boundary:
\begin{equation}
D_\mu(a,t)\supset D_k(a,\log(d_2+1)),\quad d_2=\frac 1{\gamma^{-1}(t)},
\end{equation}
and $\mu(a,b)<t$ if $k(a,b)<\log(d_2+1)$.

Also, $\zeta_k^\mu(s)=\gamma(1/(e^s-1))$.
If we put
$$
s=\log\left(\frac 1{\gamma^{-1}(t)}+1\right),\quad\mbox{ we have }
e^s-1=\frac 1{\gamma^{-1}(t)},\quad
t=\gamma\left(\frac 1{e^s-1}\right).
$$
\end{theorem}
\begin{theorem}
\cite[Theorem 3.8]{se}
If $G\subset\mathbb{R}^n$ is open, $x\in G$ and $t>0$ then
$$
D_{j}(x,t)\subset B^{n}(x,R)
$$
where $R=(e^t-1)\,d(x)$. This formula for $R$ is the best possible expressed in terms of $t$ and $d(x)$ only.
\end{theorem}

Therefore, using $d(x)\leqslant{\rm diam}(G)$, we get $\zeta_j^{|\cdot|}(t)=(e^t-1)\,{\rm diam}(G)$.

\begin{theorem}
\cite[Theorem 3.10]{se}
If $G\subset\mathbb{R}^n$ is an open set, $x\in G$ and $t>0$ then $D_{\delta}(x,t)\subset B^n(x,R)$
where $R=(e^t-1)\,d(x)$.
\end{theorem}

As above, we get $\zeta_\delta^{|\cdot|}(t)=(e^t-1)\,{\rm diam}(G)$.

From Lemma \ref{vu1l2-39}, we have that
$$
\zeta_k^\mu(t)=h(3t).
$$
Now, from \cite[Lemma 2.30]{h} we may choose (the case $n=2$)
$$
h(t)=\frac{2\pi\alpha}{\log\frac 1{2t}},\quad\mbox{for }t\leqslant \frac 14.
$$
From here we have that
$$
\zeta_k^\mu(t)=\frac{2\pi\alpha}{\log\left(\frac 1{6t}\right)},
\mbox{for }\quad t\leqslant\frac 1{12}\quad
\mbox{(more important case)}
$$
$$
\alpha=\max\{1,\gamma\}, \quad\gamma=\frac 98\log 2>1
\,,\quad\alpha=\gamma \,.
$$
In the second case, where
$$
h(t)=36\beta\pi t^2,\quad\mbox{for }t>\frac 14.
$$
we have $h(3t)=324\beta\pi t^2$, $t>\frac 1{12}$.
$$
\beta=\max\Big(1,\frac 1\gamma\Big)=1
$$
$$
\zeta_k^\mu(t)=324\pi t^2,\quad\mbox{for }t>\frac 1{12}.
$$
\vspace{0.3em}

\begin{tabular}{|c|p{3.2cm}|p{3cm}|p{4.1cm}|p{3.4cm}|}
    \hline
    & \hspace{7mm} $j_G$ & $k_G$ & $\mu_G$ & $\lambda^{-1}_G$ \\
    \hline
    \multirow{3}{1cm}{$j_G$} & 1 & 2 & 3 & 4 \\
    & $\zeta_j^j(t)=t$ & \begin{minipage}{2.5cm}$\zeta_j^k(t)=ct$\\ $G$ -- uniform\\ $\zeta_j^k(t)=\varphi(t)$\\ $G$ -- $\varphi$ domain\end{minipage} &
    \begin{minipage}{4.1cm}$\displaystyle\zeta_j^\mu(t)=\omega_{n-1}\left(\log\left(\frac 1t\right)\right)^{1-n}$\\ locally\end{minipage}
    & \begin{minipage}{3.5cm} $\displaystyle\zeta_j^{\lambda^{-1}}(t)=\frac 1{c\tau(e^{2t}-1)}$\\ $G$ -- $c$-QED domain \end{minipage} \\
    \hline
    \multirow{3}{1cm}{$k_G$} & 5 & 6 & 7 & 8 \\
    & \hspace{7mm} $\zeta_k^j(t)=t$  & $\zeta_k^k(t)=t$ & \begin{minipage}{2.8cm}$\zeta_k^\mu(t)=\gamma\left(\frac 1{e^t-1}\right)$\\
    $\partial G$ connected, nondegenerate\end{minipage} & $\zeta_k^{\lambda^{-1}}=\zeta_j^{\lambda^{-1}}$ \\
    \hline
    \multirow{3}{1cm}{$\mu_G$} & 9 & 10 & 11 & 12 \\
    & \begin{minipage}{2.5cm}$\zeta_\mu^j(t)=\frac{t}{c_n}$ \\ $\partial G$ connected\\[-0.5em]\rule{0mm}{0mm}\end{minipage} & \begin{minipage}{2.5cm}$\zeta_\mu^k(t)=c\cdot t$ \\ $G$ uniform\\[-0.5em]\rule{0mm}{0mm}$\partial G$ connected\end{minipage}
    & $\zeta_\mu^\mu(t)=t$ & \begin{minipage}{3.2cm}$\zeta_\mu^{\lambda^{-1}}=\zeta_\mu^j\circ\zeta_j^{\lambda^{-1}}$\\
    $G$ -- $c$-QED domain\\ $\partial G$ connected\end{minipage} \\
    \hline
    \multirow{3}{1cm}{$\lambda^{-1}_G$} & 13 & 14 & 15 & 16 \\
    & $\zeta_{\lambda^{-1}}^j(t)=\log\left(1+\tau^{-1}\left(\frac 1{\sqrt{2}\,t}\right)\right)$ &
    \begin{minipage}{3cm}$\zeta_{\lambda^{-1}}^k(t)=\log{\frac{1}{1-\tau_2^{-1}(1/(\sqrt{2}t))}}$\\ $t<\frac 1{\sqrt{2}\tau_2(1)}$
    \\[-0.5em] \rule{0mm}{0mm}\end{minipage} & \begin{minipage}{4.1cm}$\zeta_{\lambda^{-1}}^\mu=\zeta_{\lambda^{-1}}^j\circ\zeta_j^\mu$\\
    locally\end{minipage}
    & $\zeta_{\lambda^{-1}}^{\lambda^{-1}}(t)=t$ \\
    \hline
  \end{tabular}
\vspace{0.3em}

This is improved $4\times 4$ chart.

\begin{example}
For $G\subset\mathbb{R}^n$ we choose $z_0\in\partial G$, sequence $x_k\in G$ such $x_k\rightarrow z_0$
and sequence $y_k\in G$ such that
\begin{equation}
|y_k-z_0|<\frac{|x_k-z_0|}k.
\end{equation}
Clearly $|x_k-y_k|\rightarrow 0$ and
\begin{equation}
|x_k-y_k|>|x_k-z_0|-|y_k-z_0|>|x_k-z_0|\left(1-\frac 1k\right).
\end{equation}
But
$$
j_G(x_k,y_k)\geqslant\log\left(1+\frac{|x_k-y_k|}{|y_k-z_0|}\right)\geqslant
\log\left(1+\frac{1-\frac 1k}{\frac 1k}\right)=\log(k)\rightarrow+\infty.
$$
\end{example}
Hence $id:(G,|\cdot|)\longrightarrow(G,j_G)$ is not uniformly continuous.
By this reason, adequate fields in the chart are empty.

Also, for a fixed small $d>0$ we can find $x,y\in G$ such that $|x-y|=d$ and
$d(x,\partial G)$ as small as we like.

So we get $k_G(x,y)$ as large as we like and there is no estimate of
$k_G(x,y)$ in terms of $|x-y|$.

In other hand function $\zeta_k^{|\cdot|}$ is obtained from:
$$
k_G(x,y)\geqslant\int_0^{|x-y|}\frac{ds}{{\rm diam}(G)}=\frac{|x-y|}{{\rm diam}(G)}.
$$
From here we get that modulus of continuity of
$id:(G,k_G)\longrightarrow(G,|\cdot|)$ is $\zeta_k^{|\cdot|}(t)=t\,{\rm
diam}(G)$ (where $G$ is bounded).

All the remaining items are obtained by composition of
the above moduli of continuity.

And finally we have following charts:

\vfill

\pagebreak

\hspace{-3cm}\includegraphics[width=20cm]{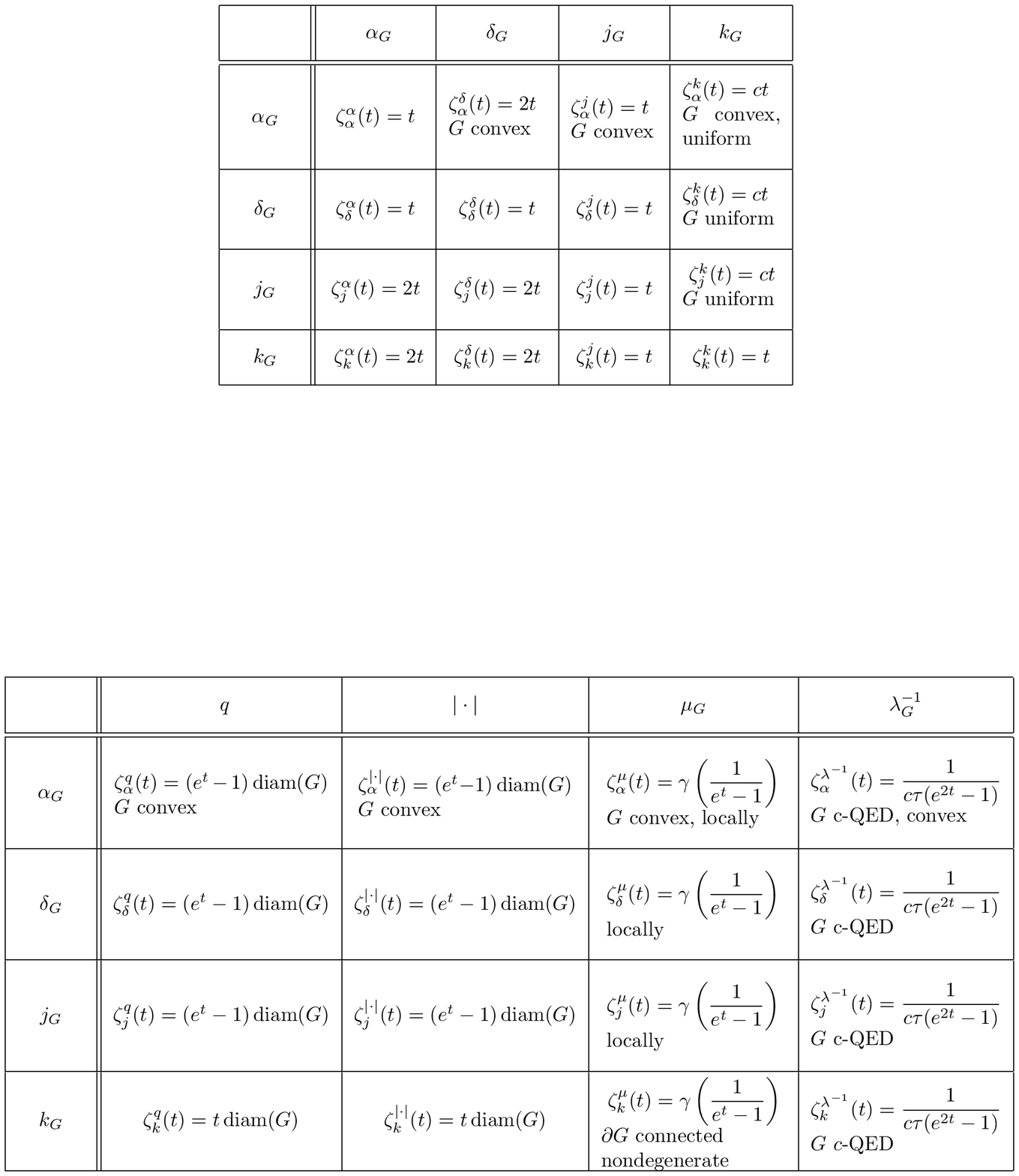}

\pagebreak

\hspace{-3cm}\includegraphics[width=20cm]{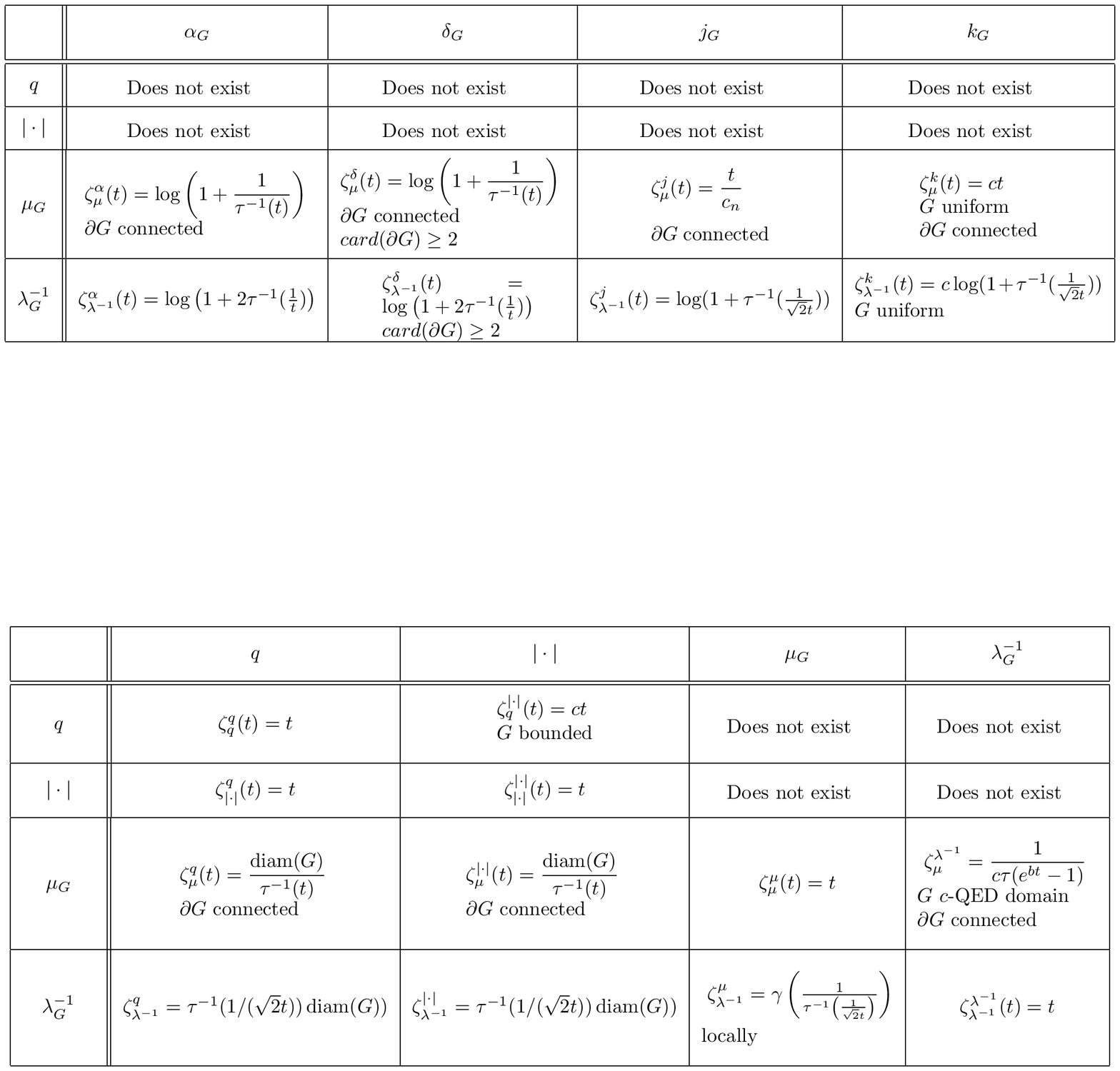}

Sharper results can be obtained for special domains, for example $G=\mathbb R^n\setminus\{0\}$ was studied
by R. Klen \cite{kl} in relation to $j_G$ metrics.

We return to the question of moduli of continuity, from a different viewpoint, in chapter 2, sections 2 and 3.

\section{Removing a point}

Let $\mathcal M$ be a collection of metrics on a domain $G\subset\mathbb R^n$ and
$B_m(x,T)=\{z\in G\,:\,m(x,z)<T\}$, $m\in\mathcal M$. Let
$$
\begin{array}{l}
r_T=\sup\{r>0\,:\,S^{n-1}(x,r)\subset B_m(x,T)\},
\vspace{0.5em}\\
R_T=\inf\{r>0\,:\,S^{n-1}(x,r)\cap B_m(x,T)=\emptyset\}.
\end{array}
$$
The question is can we find lower bound for $r_T$ and upper bound for $R_T$.

\begin{problem}\rm
(Radius of circumscribed ball)

It is evident from the definition of $\lambda_G$ that adding new
points, even isolated ones, to the boundary of $G$ will affect the
value of $\lambda_G(x,y)$ for fixed points $x,y\in G$. We study this
phenomenon in the case when $G=\mathbb R^2\setminus\{0\}$.

We find an upper bound for radius of circumscribed ball, where $m=\lambda_{G}^{-1}$.

We use notation
$$
B_\lambda(1,T)=\{z\in\mathbb C\,:\,\lambda_G(z,1)\geqslant T^{-1}\}.
$$
Let $h(z)=\frac{z}{|z|^2}$ be an inversion. Since $h:B_\lambda\longrightarrow B_\lambda$
($h$ is an isometry for $\lambda$ metric) we have
$$
\lambda_G(1,z)=\lambda_G(1,h(z)).
$$

\begin{figure}[!ht]
  \begin{center}
    \includegraphics[width=80mm]{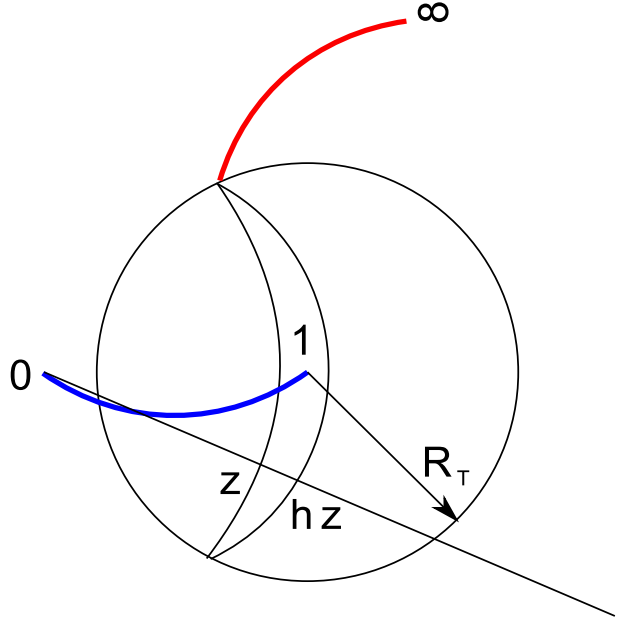}
    \caption{Radius of circumscribed ball}
  \end{center}
\end{figure}

From \cite[(3.3), (3.22)]{solv}
we have
\begin{equation}
p(z)=\frac{2\pi}{\log M(2z-1)},\quad z\in\mathbb{C}\setminus\{0,1\}\quad\mbox{and}
\end{equation}
\begin{equation}
\log M(2e^{i\theta}-1)=\frac{2\pi{\K}(\sin\frac\theta 4){\K}(\cos\frac\theta 4)}%
{{\K}^2(\sin\frac\theta 4)+{\K}^2(\cos\frac\theta 4)}.
\end{equation}
If we put $z=e^{i\theta}$ we have
$$
p(e^{i\theta})=\frac{{\K}^2(\sin\frac\theta 4)+{\K}^2(\cos\frac\theta 4)}%
{{\K}(\sin\frac\theta 4){\K}(\cos\frac\theta 4)}.
$$
For $|z|=1$ we obtain $\lambda_G(1,z)=p(z)$.

Choose $\theta$ such that $\sin\frac\theta 2=\frac{R_T}2$. From here $\theta=2\arcsin\frac{R_T}2$.
Now if we put
\begin{equation}y=\frac{{\K}(\sin\frac\theta 4)}{{\K}(\cos\frac\theta 4)}=
\label{mu}
\frac{\displaystyle 2}{\displaystyle \pi}\mu(\cos\frac\theta 4)
\end{equation}
we have
$$
p(e^{i\theta})=y+\frac 1y=\frac 1T.
$$
We are interested for solutions $y<1$ because we want $\theta<\pi$.
From here $y=\frac{2T}{1+\sqrt{1-4T^2}}$. Since from (\ref{mu})
$$
\theta=4\arccos(\mu^{-1}(\frac{\pi y}2))
$$
now we have
\begin{equation}
\label{theta}
\theta=4\arccos(\mu^{-1}\big(\frac\pi 2\frac{2T}{1+\sqrt{1-4T^2}}\big))=
4\arccos(\mu^{-1}\big(\frac{\pi T}{1+\sqrt{1-4T^2}}\big)).
\end{equation}

Hence, the radius of the circumscribed sphere is
$$
R_T=2\sin\frac\theta 2,\quad T\in(0,\frac 12),\quad \theta \mbox{ from (\ref{theta})}.
$$
\end{problem}
\begin{question}\rm

\begin{enumerate}
\item
Can we find $r_T$ in the case above?
\item
Can we estimate $R_T$, where $G$ is now bounded subset of $\mathbb C$ (instead of
$\mathbb R^2\setminus\{0\}$)?
\item
Consider $\mu_G$-balls where $\partial G$ is connected, say
$\partial G=[0,e_1]$. Can we find a lower bound for $r_T$ (upper
bound for $R_T$) in this case?
\end{enumerate}
\end{question}
\begin{problem}\rm
(Estimate for $\lambda_{B^2\setminus\{0\}}(x,y)$)
Next we investigate the following situation: $G\subseteq\mathbb R^n$ is domain,
$a\in G$, $G'=G\setminus\{a\}$. Is $\lambda_G(x,y)=\lambda_{G'}(x,y)$ true under some
additional assumptions, like $x,y$ close to $\partial G$?

We consider a special case where $G=B^2$ and $a=0$.
\begin{figure}[!ht]
\begin{center}
\includegraphics[width=40mm]{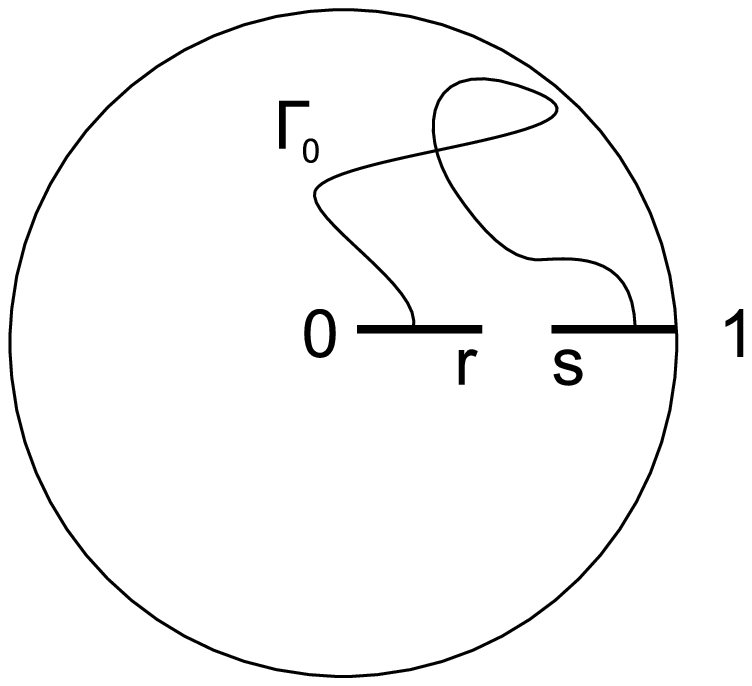}
\caption{}
\end{center}
\end{figure}

In \cite[Lemma 2.8]{levu} is proven that if $\Gamma_0=\Delta([0,x],[\tilde{y},x/|x|];B)$, where $\tilde{y}=\frac{|y|}{|x|}\,x$ and if we put
$|x|=r$, $|\tilde{y}|=s$, then we have
\begin{equation}
\label{m_gamma_0}
M(\Gamma_0)=\tau\left(\frac{(s-r)(1-rs)}{r(1-s)^2}\right).
\end{equation}

Further, from \cite[(2.6)]{vu1} we have that if
$\Delta_0=\Delta([-\frac{x}{|x|},-x],[x,\frac{x}{|x|}];B)$ and if $|x|=r$
as before, then
\begin{figure}[!ht]
\begin{center}
\includegraphics[width=40mm]{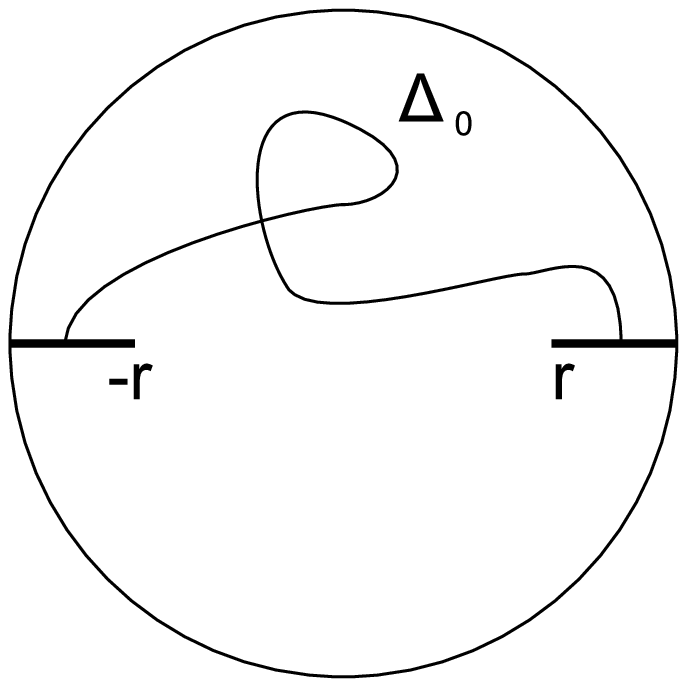}
\caption{}
\end{center}
\end{figure}
$$
M(\Delta_0)=\frac{1}{2}\,\tau\left(\frac{4r^2}{(1-r^2)^2}\right).
$$
Also, using M\" obius transformation $T_r:B^2\longrightarrow B^2$, $T(r)=0$ we can map family of curves
$\Delta_1$ to family of curves $\Delta_1'$, where $\Delta_1=\Delta([-\frac{x}{|x|},-\tilde{y}],[0,x];B)$ and
$\Delta_1'=\Delta([-\frac{x}{|x|},-\tilde{y}'],[-x,0];B)$.

We know that
$$
\rho(-s,0)=\rho(-r,-t),
$$
where $r$ and $s$ are as before and $-t=T_r(-s)$. Further, this is equivalent to
\begin{equation}
\label{log}
\log\frac{1+s}{1-s}=\log\frac{1+t}{1-t}\frac{1-r}{1+r}.
\end{equation}
Solving (\ref{log}) in $t$ we obtain $t=\frac{s+r}{1+sr}$.

Now we have
$$
M(\Delta_1)=M(\Delta_1')=\tau\left(\frac{(t-r)(1-tr)}{r(1-t)^2}\right)
=\tau\left(\frac{s(1+r)^2}{r(1-s)^2}\right).
$$

\begin{figure}[!ht]
\begin{center}
\includegraphics[width=100mm]{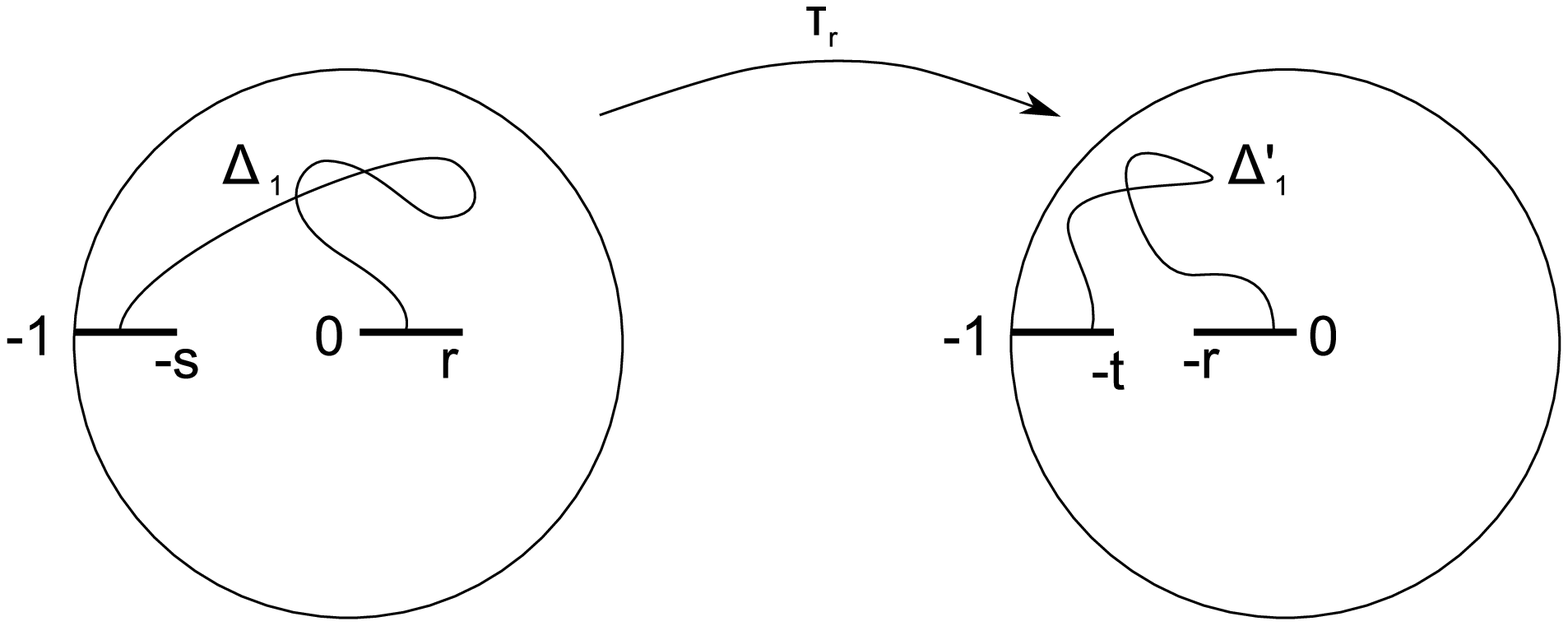}
\caption{}
\end{center}
\end{figure}

The first equality holds because $T_r$ is conformal map, the second
one follows from (\ref{m_gamma_0}) and the third one from the expression for
$t$.

Now, if we put in last term that $r=s$, we obtain
$$
M(\Delta_1)=\tau\left(\left(\frac{1+r}{1-r}\right)^2\right).
$$

The question is when is $M(\Delta_1)\geqslant M(\Delta_0)$. In other words, when is
\begin{equation}
\label{frac_12_tau}
\tau\left(\left(\frac{1+r}{1-r}\right)^2\right)\geqslant\frac 12\tau\left(\frac{4r^2}{(1-r^2)^2}\right)?
\end{equation}

Applying formula \cite[5,19 (5)]{avv}:
$$
\frac 12\tau(t)\geqslant\tau((\sqrt{t}+\sqrt{t+1})^4-1)
$$
for $t=4r^2/(1-r^2)^2$ we have
$$
\frac 12\tau\left(\frac{4r^2}{(1-r^2)^2}\right)=\tau\left(\frac{8r(r^2+1)}{(1-r)^4}\right).
$$
Then (\ref{frac_12_tau}) is equivalent to
$$
\left(\frac{1+r}{1-r}\right)^2\leqslant
\frac{8r(r^2+1)}{(1-r)^4},
$$
since $\tau$ is decreasing. The last inequality is equivalent to
$$
r^4-8r^3-2r^2-8r+1\leqslant 0.
$$

This inequality holds for $r\in[0.12,1)$.

This gives the answer to the question: For which values of $|x|$ we have
$$
\lambda_A(x,-x)=M(\Delta(E,-E;B^2)),
$$
where $A=B^2\setminus\{0\}$, $E=[x,\frac{x}{|x|}]$?

A related result can be found in Heikkala's dissertation,
\cite[Theorem 7.3]{h}. In fact, this theorem deals with the more
general situation: If $x$ and $y$ are close to the boundary and far
apart then $\lambda_{B^n\setminus\{0\}}(x,y)=\lambda_{B^n}(x,y)$.
His theorem is:
\begin{theorem}
Let $G=B^n\setminus\{0\}$ and let $x,y\in G$ with $|x-y|\ge\delta>0$. Then,
if $\min\{|x|,|y|\}\in(r_1,1)$ with $r_1=\frac{\sqrt{\delta^4+64}-\delta^2}{8}$, we have that
$$
\lambda_G(x,y)=\lambda_{B^n}(x,y).
$$
\end{theorem}

However, we have in the special case $x=-y$, better constant (letting $\delta=2|x|$ and
$r_1=|x|$ in Theorem 7.3 gives equation $r_1^3+r_1^{2}-1=0$, and its real root is
larger than $0.75$, and consequently larger than $0.12$).
\end{problem}

\section{Uniform continuity on union of two domains}

\begin{definition}
Let $\{m_D\,:\,D\subseteq\overline{\mathbb R^n}\}$ be a family of metrics.
We say that this family is monotone if $D_1\subseteq D_2$ implies
$m_{D_1}(x,y)\ge m_{D_2}(x,y)$ for all $x,y\in D_1$.
\end{definition}

\begin{lemma}\cite[2.27]{vu5}
\label{lemma}
Let $G_1,G_2$ be domains in $\mathbb R^n$ with $G_1\cap G_2\neq\emptyset$,
$G_1\neq\mathbb R^n\neq G_2$ and assume that there exists $c\in(0,1)$ such that
\begin{equation}
\label{condition}
d(x,\partial G_1)+d(x,\partial G_2)\ge c\,d(x,\partial(G_1\cup G_2)),
\end{equation}
for all $x\in G=G_1\cup G_2$.

Suppose that $f:G\longrightarrow fG$ is continuous, $fG\subseteq\mathbb R^n$;
that $\{m_D\,:\,D\subseteq\mathbb R^n\}$ is a monotone family of metrics;
and that
\begin{equation}
\label{monotonicity_1}
m_{fG_j}(f(x),f(y))\le\omega_j(k_{G_j}(x,y))
\end{equation}
for $x,y\in G_j$ and $j=1,2$. Then there exists
$\omega:[0,+\infty)\longrightarrow[0,+\infty)$ such that
\begin{equation}
\label{monotonicity_k}
m_{fG}(f(x),f(y))\le\omega(k_G(x,y))
\end{equation}
and $\lim_{t\rightarrow 0+}\omega(t)=0$ provided $\lim_{t\rightarrow 0+}\omega_j(t)=0$,
$j=1,2$.
\end{lemma}

Now we consider a similar, but local result, with $j$ metric replacing $k$
metric. We can no longer use geodesics as was done in the proof of
the above lemma.

\begin{lemma}
\label{an_lemma}
Let $G_1,G_2$ be domains in $\mathbb R^n$ with $G_1\cap G_2\neq\emptyset$,
$G_1\neq\mathbb R^n\neq G_2$ and assume that there exists $c\in(0,1)$ such that
$$
d(x,\partial G_1)+d(x,\partial G_2)\ge c\,d(x,\partial(G_1\cup G_2)),
$$
for all $x\in G=G_1\cup G_2$.

Suppose that $f:G\longrightarrow fG$ is continuous, $fG\subseteq\mathbb R^n$;
that $\{m_D\,:\,D\subseteq\mathbb R^n\}$ is a monotone family of metrics;
and that
$$
m_{fG_j}(f(x),f(y))\le\omega_j(j_{G_j}(x,y))
$$
for $x,y\in G_j$ and $j=1,2$. Then there exists
$\omega:[0,\delta)\longrightarrow[0,+\infty)$, where $\delta=\log\left(1+\frac c4\right)$ such that
\begin{equation}
\label{monotonicity_j}
m_{fG}(f(x),f(y))\le\omega(j_G(x,y))
\end{equation}
for $x,y\in G$, $j_G(x,y)\le\delta$ and $\lim_{t\rightarrow 0+}\omega(t)=0$ provided $\lim_{t\rightarrow 0+}\omega_j(t)=0$,
$j=1,2$.
\end{lemma}

\begin{proof}
Let $d(x)=d(x,\partial G)$ and $j_G(x,y)\le\delta$.

Then, we have $|x-y|\le\frac c4\min\{d(x),d(y)\}$. We may assume
$d(x)\le d(y)$. By the hypothesis (\ref{condition}) of the lemma
there exists $i\in\{1,2\}$ such that $d(x,\partial G_i)\ge \frac
c2d(x)$, i.e., $B^n(x,c\,d(x)/2)\subseteq G_i$. Without loss of
generality, we may assume that $i=1$. Then
$$
y\in B^n(x,c\,\min\{d(x),d(y)\}/4)\subseteq B^n(x,\frac 12d(x,\partial G_1)).
$$
We have
$$
j_{G_1}(x,y)=\log\left(1+\frac{|x-y|}{\min\{d_1(x),d_1(y)\}}\right)
$$
where $d_1(z)=d(z,\partial G_1)$. By the above calculation, $|x-y|\le\frac 12\,d_1(x)$
and hence $d_1(y)\ge\frac 12\,d_1(x)$. The last inequality now yields
$$
\begin{array}{rcl}
j_{G_1}(x,y) & \le & \displaystyle\log\left(1+\frac{2|x-y|}{d_1(x)}\right)
\vspace{0.5em}\\
& \le & \displaystyle\log\left(1+\frac{4|x-y|}{c\,d(x)}\right)
\vspace{0.5em}\\
& < & \displaystyle \frac 4c\,j_G(x,y).
\end{array}
$$

Conclusion:
\begin{equation}
\label{case_a}
m_{fG}(f(x),f(y))\le m_{fG_1}(f(x),f(y))\le\omega_1(j_{G_1}(x,y))
\le\omega_1\left(\frac 4c\,j_G(x,y)\right)
\end{equation}
where also monotone property of the family $\{m_D\}$ was applied.
\end{proof}

\begin{example} {\label{ferrexam}} {\rm
We present an example, due to J. Ferrand, of two domains
$G_1,G_2\subseteq\mathbb C$, $G_1,G_2\neq\emptyset$ and an analytic
function $f:H\longrightarrow\mathbb C$, $H=G_1\cup G_2$ such that
\begin{enumerate}
 \item (\ref{monotonicity_1}) holds in $G_1$ and $G_2$.
\item (\ref{monotonicity_k}) does not hold on $H$.
\end{enumerate}

We set $G_1=\mathbb C\setminus\{p+iq\,:\,p,q\in\mathbb Z\}$ and
$G_2=\mathbb C\setminus(\{0\}\cup\{p+1/2+iq\,:\,p,q\in\mathbb Z\})$.

Note that $G_1\cap G_2\neq\emptyset$ and $G_1\cup G_2=H=\mathbb C\setminus\{0\}$.
We define $f(\xi)=e^{4\pi\xi}$. This is an entire function.

It is easy to see that $f(G_1)=f(G_2)=f(H)=H$. In fact $f(\Omega_k)=H$,
where $\Omega_k=\{x+iy\,:\,k<y<k+1\}$. Quasihyperbolic distance in $H$
satisfies
\begin{equation}
\label{infimum}
k_H(w_1,w_2)=\inf_{e^{z_1}=w_1,e^{z_2}=w_2}|z_1-z_2|.
\end{equation}
Also, for $i=1,2$ holds $d(\xi,\partial G_i)\leqslant 1/2$, so
the metric density $\frac{1}{d(\xi,\partial G_i)}$ exceeds $\sqrt{2}$
and therefore (by a line integration)
\begin{equation}
\label{sqrt_2_xie}
k_{G_i}(\xi_1,\xi_2)\geqslant\sqrt{2}|\xi_1-\xi_2|.
\end{equation}
Now, (\ref{infimum}) tells us that $f:G_i\longrightarrow H$ is Lipschitz
with respect to euclidean metric in $G_i$ and quasihyperbolic metric in $H$,
and by (\ref{sqrt_2_xie}) it is also Lipschitz with respect to quasihyperbolic metric
in $H$ and $G_i$.

But $f$ is not uniformly continuous as a map $(H,k_H)\longrightarrow(H,k_H)$:
in fact we have
$$
\lim_{n\rightarrow\infty}k_H(n,n+1)=\log\frac{n+1}{n}=0,
$$
while
$$
\lim_{n\rightarrow\infty}k_H(f(n),f(n+1))=\log\frac{e^{4\pi(n+1)}}{e^{4\pi n}}=4\pi.
$$

Note that our domains fail to meet condition (\ref{infimum}) from
\cite[Lemma 2.27 ]{vu5}. Indeed for large $|x|$ we have
$$
d(x,\partial(G_1\cup G_2))=|x|
$$
and
$$
d(x,\partial G_1)+d(x,\partial G_2)\leqslant 2\frac{1}{\sqrt{2}}=\sqrt{2},
$$
so there is no $c\in(0,1)$ such that (\ref{infimum}) is valid. }
\end{example}


\begin{remark}
\begin{enumerate}
\item
It would be of interest to find a homeomorphism $f$ with properties
as in Example \ref{ferrexam} due to Ferrand.
\item
It is natural to expect that there is a counterpart of Lemma
\ref{an_lemma} for other metrics in place of $j$.
\item
Is the condition (\ref{condition}) invariant under the
quasiconformal mappings?
\end{enumerate}
\end{remark}

\section{Quasiconformal maps with identity boundary values}

For a domain $G\subset\mathbb R^n$, $n\geqslant 2$, let
$$
Id(\partial G)=\{f\,:\,\overline{\mathbb R^n} \to \overline{\mathbb
R^n} \mbox{ homeomorphism }:\,f(x)=x,\quad\forall x\in\overline{\mathbb
R^n}\setminus G\}.
$$
Here $\overline{\mathbb R^n}$ stands for the M\"obius space $\mathbb
R^n \cup \{ \infty \} \,.$ We shall always assume that $card  \{
\overline{\mathbb R^n} \setminus G \} \ge 3.$
 If $K\geqslant 1$, then the class of $K$-quasiconformal maps in
$Id(\partial G)$ is denoted by $Id_K(\partial G)$. Here we use
 notation and terminology from V\"ais\"al\"a's book
\cite{v2}. In particular, $K$-quasiconformal maps are defined in
terms of the maximal dilatation as in \cite[p. 42]{v2} if not
otherwise stated.

We will study the following well-known problem:

\begin{problem}\rm \label{myprob}
\begin{enumerate}
\label{prob1}
\item
Given $a,b \in G$ and  $f\in Id(\partial G)$ with $f(a)=b,$ find a
lower bound for $K(f)$.
\item
\label{prob1-2} Given $a,b\in G,$ construct $f\in Id(\partial G)$
with $f(a)=b$ and give an upper bound for $K(f)$.
\end{enumerate}
\end{problem}


O. Teichm\"uller studied this problem in the case when $G$ is a
plane domain with $card(\overline{\mathbb R^2}\setminus G)=3$ and
proved the following theorem with a sharp bound for $K(f)$.

\begin{theorem}\label{thm1}
Let $G= \mathbb{R}^2 \setminus \{0,1\}$, $a,b\in G$. Then there
exists $f\in Id_K(\partial G)$ with $f(a)=b$ iff
$$
\log(K(f))\geqslant s_G(a,b),
$$
where $s_G(a,b)$ is the hyperbolic metric of $G$.
\end{theorem}

\begin{theorem}
\label{main_theorem} If $f\in Id_K(\partial B^n)$, then for all
$x\in B^n$
$$
\rho_{B^n}(f(x),x)\leqslant\log\frac{1-a}a,\quad
a=\varphi_{1/K,n}(1/\sqrt{2})^2,
$$
where $\varphi_{K,n}$ is as in (\ref{phindef}).
\end{theorem}

\begin{theorem} \label{mycor}
If $f\in Id_K(\partial B^n)$, then for all $x\in B^n, n\ge 2,$ and
$K\in[1,17]$
\begin{equation} \label{luckynumber}
  |f(x) -x| \le \frac{9}{2} (K-1) \,.
\end{equation}
 For $n=2$ we have
\begin{equation} \label{n=2}
|f(x)-x|\leqslant\frac{b}{2}(K-1),\quad b\leqslant 4.38.
\end{equation}
\end{theorem}


The theory of $K$-quasiregular mappings in ${\mathbb R}^n, n \ge 3,$
with maximal dilatation $K$ close to $1\,$ has been extensively
studied by Yu. G. Reshetnyak \cite{r} under the name "stability
theory". By Liouville's theorem we expect that when  $n \ge 3$ is fixed and
$K\to 1$ the $K$-quasiregular maps "stabilize", become more and more
like M\"obius transformations, and this is the content of the deep
main results of \cite{r} such as \cite[p. 286]{r}. We have been
unable to decide whether Theorem \ref{main_theorem} follows from
Reshetnyak's stability theory in a simple way.
V. I. Semenov \cite{s} has also made significant contributions to this
theory. For the plane case P. P. Belinskii  has found several
sharp results in \cite{bel}.

\begin{problem}
\rm
It seems possible that there is a new kind of
stability behavior: If $K>1$ is fixed, do maps in $Id_K(\partial
B^n)$ approach identity when $n \to \infty$? Our results do not
answer this question. This kind of behavior is anticipated in
\cite[Open problem 9, p. 478]{avv}.
\end{problem}
\begin{lemma} \label{hypmetr}
For $x,y\in B^n$ let $t=\sqrt{(1-|x|^2)(1-|y|^2)}$. Then for $x,y\in
B^n$
\begin{equation}
\tanh^2\frac{\rho_{B^n}(x,y)}2=\frac{|x-y|^2}{|x-y|^2+t^2}\, ,
\end{equation}
\begin{equation}
\label{tanh} |x-y|\leqslant 2\tanh\frac{\rho_{B^n}(x,y)}{4}=
\frac{2|x-y|}{\sqrt{|x-y|^2+t^2}+t}\, ,
\end{equation}
where equality holds for $x=-y$.
\end{lemma}

Next, we consider a decreasing
homeomorphism $\mu:(0,1)\longrightarrow(0,\infty)$ defined by
\begin{equation}
\label{dec_hom} \mu(r)=\frac\pi 2\,\frac{{\K}(r')}{{\K}(r)}, \quad
{\K}(r)=\int_0^1\frac{dx}{\sqrt{(1-x^2)(1-r^2x^2)}}\, ,
\end{equation}
where ${\K}(r)$ is Legendre's complete elliptic integral of the
first kind and $ r'=\sqrt{1-r^2},$ for all $r\in(0,1)$.
The Hersch-Pfluger distortion function is an increasing
homeomorphism $\varphi_K:(0,1)\longrightarrow(0,1)$ defined by
\begin{equation} \label{phidef}
\varphi_K(r)=\mu^{-1}(\mu(r)/K)
\end{equation}
for all $r\in(0,1)$, $K>0$. By continuity we set $\varphi_K(0)=0$,
$\varphi_K(1)=1$. From (\ref{dec_hom}) we see that
$\mu(r)\mu(r')=\left(\frac{\pi}2\right)^2$ and from this we are able
to conclude a number of properties of $\varphi_K$. For instance, by
\cite[Thm 10.5, p. 204]{avv}
\begin{equation} \label{phipyth}
\varphi_K(r)^2+\varphi_{1/K}(r')^2=1,\quad r'=\sqrt{1-r^2},
\end{equation}
holds for all $K>0$, $r\in(0,1)$.

\begin{subsec}{\bf Special function $\varphi_{K,n}$}
We use the standard notation
\begin{equation} \label{phindef}
\varphi_{K,n}(r)=\frac 1{\gamma_n^{-1}(K \gamma_n(1/r))}.
\end{equation}
 Then $\varphi_{K,n}:(0,1)\longrightarrow(0,1)$ is an increasing homeomorphism,
see \cite[(7.44)]{vu2}. Because $\gamma_2(1/r) = 2 \pi/\mu(r)$ by
\cite[(5.56)]{vu2}, it follows that $\varphi_{K,2}(r)$ is
the same as the $\varphi_{K}(r)$ in (\ref{phidef}).
\end{subsec}

\begin{subsec}{\bf
The key constant.} The special functions introduced above will have
a crucial role in what follows. For the sake of easy reference we
give here some well-known identities between them that can be found
in \cite{avv}. First, the function
\begin{equation}
\eta_{K,n}(t)= \tau_n^{-1}(\tau_n(t)/K) =
\frac{1-\varphi_{1/K,n}(1/\sqrt{1+t})^2}{\varphi_{1/K,n}(1/\sqrt{1+t})^2},
\, K
>0\,,
\end{equation}
defines an increasing homeomorphism $\eta_{K,n}: (0,\infty)\to (0,\infty)\,$(cf.
\cite[p.193]{avv}). The constant $(1-a)/a,
a=\varphi_{1/K,n}(1/\sqrt{2})^2,$ in (\ref{luckynumber}) can be
expressed as follows for $K>1$
\begin{equation}
\label{aform} (1-a)/a = \eta_{K,n}(1) =\tau_n^{-1}(\tau_n(1)/K) \, .
\end{equation}
Furthermore, by (\ref{phipyth}) 
\begin{equation} \label{eta1}
  \eta_{K,2}(t) = \frac{s^2}{1-s^2} ,
  \quad s= \varphi_{K,2}(\sqrt{t/(1+t)})\,
\end{equation}
and
\begin{equation} \label{eta1b}
  \eta_{K,2}(1) \in (e^{\pi(K-1)} , e^{b(K-1)})
\end{equation}
where $b= (4/\pi) {\K}(1/\sqrt{2})^2 = 4.376879...$ Note that the
constant $\lambda(K)$ in \cite[10.33 p. 218.]{avv} is the same as
$\eta_{K,2}(1)\,.$
\end{subsec}

For the proof of Lemma \ref{stabrmk}, we record a lower        
bound for $\varphi_{1/K,n}(r)\,.$ The constant $\lambda_n$ is the so
called Gr\"otzsch ring constant, see \cite{avv}.

\begin{lemma} (\cite[7.47, 7.50]{vu2}) \label{cgqm747} For $n \ge 2, K \ge
1,$ and $0\le r \le 1$
\begin{equation}
\varphi_{1/K,n}(r) \ge \lambda_n^{1- \beta} r^{\beta}, \, \, \beta=
K^{1/(n-1)}, \label{7472}
\end{equation}
\begin{equation}  \label{7502}
\lambda_n^{1- \beta}\ge 2^{1-\beta} K^{-\beta} \ge
2^{1-K} K^{-K} \,.
\end{equation}
\end{lemma}

\begin{lemma}
\label{mn_lemma}
\begin{enumerate}
\item
For all $m,n\geqslant 1$ there is $M>1$ such that the inequality
\begin{equation}
\label{lemas_inequality}
\log(2^{mx-m+1}x^{nx}-1)\leqslant(2m\log 2+2n)(x-1)
\end{equation}
holds for $x\in[1,M]$ with equality only for $x=1$. Moreover, with
$t= (m\log 2-n)/(2n)\,,$ $M$ can be chosen as
$$
M=\sqrt{\frac{(m-1)\log 2+\log\left(1+\frac{(n+m\log
2)^2}{n}\right)}{n}+ t^2}-t.
$$
\item
Let $p(x)=\log(2^{mx-m+1}x^{nx}-1)$, $q(x)=(2m\log 2+2n)(x-1)$ and
let us use the above notation.
Let $a_0=M$ and $a_{n+1}=p^{-1}(q(a_n))$ for $n\geqslant 1$.
Then the sequence $a_n$ is increasing and bounded.
If $a=\lim_{n\rightarrow}a_n$ then the inequality (\ref{lemas_inequality})
holds for $x\in[1,a]$ with equality iff $x\in\{1,a\}$.
For $m=3$ and $n=2$ we have $a>17$.
\end{enumerate}
\end{lemma}
\begin{proof}
Let
$$
u(x)=(mx-m+1)\log 2+nx\log x,\quad v(x)=
\log(e^{u(x)}-1)=\log(2^{mx-m+1}x^{nx}-1).
$$
Then we have
$$
\begin{array}{rcl}
v''(x) & = &
\displaystyle
(\log(e^{u(x)}-1))'' =
\left(
\frac{u'(x)\,e^{u(x)}}{e^{u(x)}-1}
\right)'
\vspace{1em}\\
& = &
\displaystyle
\frac{(u''(x)e^{u(x)}+(u'(x))^2e^{u(x)})(e^{u(x)}-1)-(u'(x)\,e^{u(x)})^2}{(e^{u(x)}-1)^2}
\vspace{1em}\\
& = &
\displaystyle
\frac{e^{u(x)}}{(e^{u(x)}-1)^2}\cdot
((u''(x)+(u'(x))^2)(e^{u(x)}-1)-(u'(x))^2e^{u(x)})
\vspace{1em}\\
& = &
\displaystyle
\frac{e^{u(x)}}{(e^{u(x)}-1)^2}\cdot
(u''(x)(e^{u(x)}-1)-(u'(x))^2).
\end{array}
$$
Thus
$$
v''(x)\leqslant 0\,\,\,\Leftrightarrow\,\,\,u''(x)(e^{u(x)}-1)\leqslant(u'(x))^2.
$$
Since
$$
e^{u(x)}=2^{mx-m+1}x^{nx},\quad u'(x)=n+m\log 2+n\log x,\quad u''(x)=\frac nx,
$$
we have
$$
v''(x)\leqslant 0\,\,\,\Leftrightarrow\,\,\,\frac nx(2^{mx-m+1}x^{nx}-1)\leqslant(n+m\log 2+n\log x)^2,
$$
therefore $v''(x)\leqslant 0$ is for $x\geqslant 1$ equivalent to
$$
2^{mx-m+1}x^{nx}-1\leqslant\frac xn(n+m\log 2+n\log x)^2.
$$
Let $f(x)=2^{mx-m+1}x^{nx}-1$ and $g(x)=\frac xn(n+m\log 2+n\log x)^2$.
Both functions $f$ and $g$
are increasing on $[1,+\infty)$ and $f(1)<g(1)$ because
$$
f(1)=1\leqslant n=\frac 1n\cdot n^2<\frac 1n(n+m\log 2)^2=g(1).
$$
By continuity of $f$ we can conclude that there is $M>1$ such that
$f(M)\leqslant g(1)$. For such $M$
$$
f(x)\leqslant f(M)\leqslant g(1)\leqslant g(x),\quad x\in[1,M].
$$
This implies that $v$ is concave on $[1,M]$ and consequently
$$
v(x)\leqslant v(1)+v'(1)(x-1),\quad x\in[1,M]
$$
i.e.
$$
\log(2^{mx-m+1}x^{nx}-1)\leqslant(2m\log 2+2n)(x-1),\quad x\in[1,M].
$$
The inequality $f(x)\leqslant g(1)$ is equivalent to
\begin{equation}
\label{log_form}
(mx-m+1)\log 2+nx\log x\leqslant\log\left(1+\frac{(n+m\log 2)^2}n\right).
\end{equation}
Because
\begin{equation}
\label{linearization}
(mx-m+1)\log 2+nx\log x\leqslant(mx-m+1)\log 2+nx(x-1)
\end{equation}
the inequality (\ref{log_form}) is the consequence of the inequality
\begin{equation}
\label{quad_form}
(mx-m+1)\log 2+nx(x-1)\leqslant\log\left(1+\frac{(n+m\log 2)^2}n\right).
\end{equation}
In (\ref{linearization}) equality holds only for $x=1$. Because
$$
1+\frac{(n+m\log 2)^2}n>1+\frac{n^2}{n}=1+n\geqslant 2
$$
the inequality (\ref{quad_form}) is a strict inequality for $x=1$.
By this reason, the greater root of the quadratic equation
$$
(mx-m+1)\log 2+nx(x-1)=\log\left(1+\frac{(n+m\log 2)^2}n\right)
$$
is greater than $1$. If we denote this root with $M$ the inequality
(\ref{log_form}) holds for $x\in[1,M]$ with equality only for $x=1$.
The first part of Lemma is proved.

Now we prove the second part of the inequality. Both of functions
$p(x)$ and $q(x)$ are continuous and increasing. Consequently
$r(x)=p^{-1}(x)$ is continuous and increasing. Because
$$
p(a_1)=q(a_0)>p(a_0)
$$
using monotonicity of $p(x)$ we can conclude that $a_1>a_0$.
Now, by induction and monotonicity of $r$ we can conclude that
the sequence $a_n$ is increasing. Now for $x\in[a_n,a_{n+1})$ we have
$$
p(x)<p(a_{n+1})=q(a_n)\leqslant q(x).
$$
So $p(x)<q(x)$ holds for $x\in\bigcup_{n=0}^\infty[a_n,a_{n+1})=[a_0,a)$
and using already proved inequality, $p(x)<q(x)$ holds for $1<x<a$.
For $x\geqslant 1$ holds $mx-m+1>1$ and $x^{nx}\geqslant 1$ and consequently
$$
p(x)=\log(2^{mx-m+1}x^{nx}-1)>\log(2\,x^{nx}-1)\geqslant nx\log x.
$$
Because $p(x)>nx\log x\geqslant(n\log x)(x-1)$ inequality $p(c)>q(c)$
holds for $c$ such that $n\log c\geqslant 2m\log 2+2n$. It is
easy to see that it is true for $c=2^{\frac{2m}{n}}e^2$.
It implies that $a$ is finite (for example $a<2^{\frac{2m}{n}}e^2$) and $a_n$ is bounded.
Letting $n\rightarrow\infty$ in $p(a_{n+1})=q(a_n)$ and using continuity
of both functions we conclude that $p(a)=q(a)\,.$
\end{proof}
\begin{lemma} \label{stabrmk}
If $a=\varphi_{1/K,n}(1/\sqrt{2})^2$ is as in Theorem
\ref{main_theorem} then for $M>1$ and $\beta\in[1,M]$
\begin{equation} \label{nbnd}
\log\left(\frac{1-a}a\right)\le
\log(\lambda_n^{2(\beta-1)}2^{\beta}-1)\le V(n)(\beta-1)
\end{equation}
with $V(n)=(2\log(2\lambda_n^2))(2\lambda_n^2)^{M-1}$ and for
$K\in[1,17]$,
\begin{equation}
\label{remark_inequality}
\log\left(\frac{1-a}a\right)\leqslant(K-1)(4+6\log 2) <9(K-1),\quad
\end{equation}
with equality only for $K=1$. For $n=2$
\begin{equation} \label{ineq2}
\log\left(\frac{1-a}a\right)=
\log\left(\frac{\varphi_{K,2}(1/\sqrt{2})^2}{\varphi_{1/K,2}(1/\sqrt{2})^2}\right)
\leqslant b(K-1)
\end{equation}
where $b= (4/\pi){\K}(1/\sqrt{2})^2\le 4.38 \,.$
\end{lemma}

\begin{proof}
For $\beta\in[1,M]$ we have by (\ref{7472})
$$
\log\left(\frac{1-a}a\right)\le
\log(\lambda_n^{2(\beta-1)}2^\beta-1) \,.
$$
Further, we have
$$
\frac{\log(\lambda_n^{2(\beta-1)}2^\beta-1)}{\beta-1}\leqslant
2\,\frac{(2\lambda_n^2)^{\beta-1}-1}{\beta-1}\leqslant
(2\log(2\lambda_n^2))(2\lambda_n^2)^{M-1}.
$$
The second inequality follows from the inequality $\log(t)\leqslant
t-1$ and the third one from Lagrange's theorem and monotonicity of
the function $(2\log(2\lambda_n^2))(2\lambda_n^2)^{x-1}$. This
proves (\ref{nbnd}).

From (\ref{7502}) it follows that the constant $a$ satisfies the
inequality
$$
a\ge 2^{2(1-K)} K^{-2K} (1/\sqrt{2})^{2K} \,
$$
and also
$$
1/a \le 2^{3K-2} K^{2K}\,,\quad K>1.
$$
By Lemma \ref{mn_lemma} we have
$$
\log(2^{3K-2} K^{2K}-1)\leqslant(4+6\log 2)(K-1)
$$
for $K\in[1,17]$ with equality only for $K=1$. Now, from
$$
\frac{1-a}{a}<2^{3K-2} K^{2K}-1,\quad K>1
$$
we conclude that
$$
\log\left( \frac{1-a}{a} \right) \leqslant(4+6\log 2)(K-1) <9(K-1)
\, .
$$

For the case $n=2$ we can apply the identity (\ref{eta1}) and
the inequality in (\ref{eta1b}).
\end{proof}

\begin{figure}[!ht]
\begin{center}
\includegraphics[width=35mm]{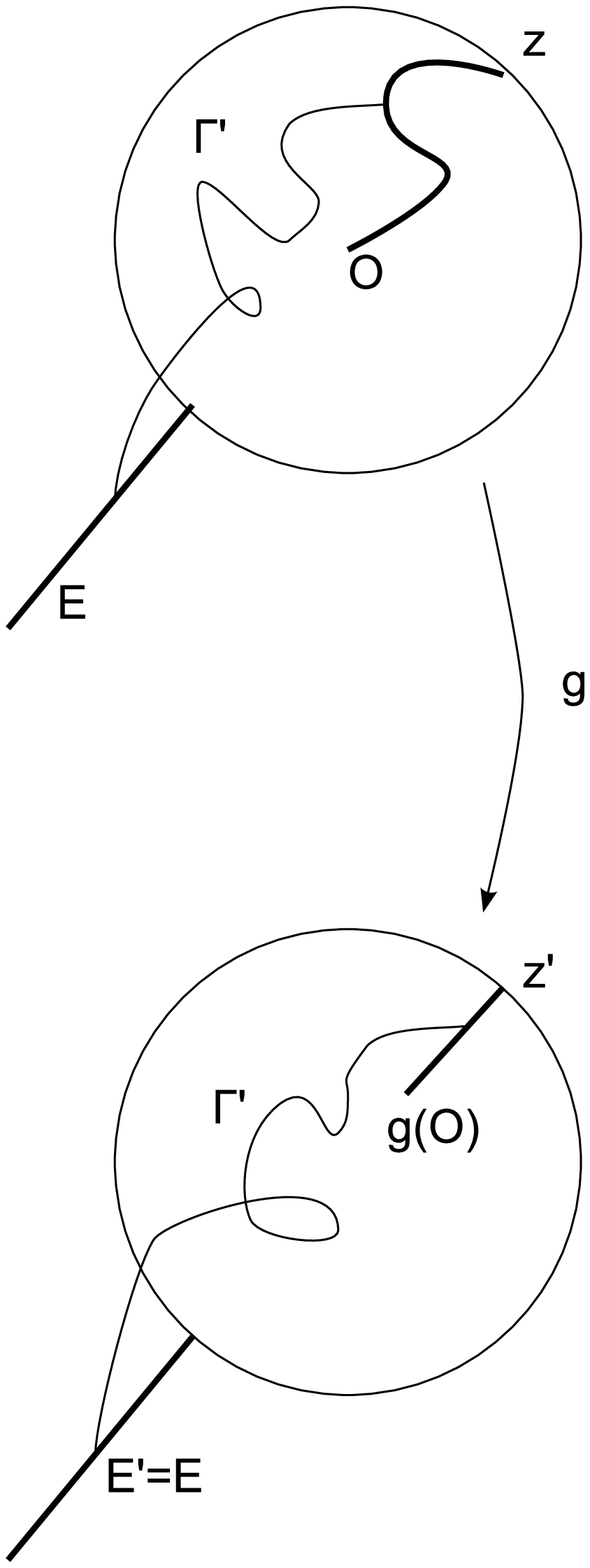}
\caption{}
\end{center}
\end{figure}

\begin{subsec}{\bf
Proof of Theorem \ref{main_theorem}.}
Fix $x\in B^n$ and let $T_x$
denote a M\" obius transformation of $\overline{\mathbb R^n}$ with
$T_x(B^n)=B^n$ and $T_x(x)=0$. Define $g:\mathbb
R^n\longrightarrow\mathbb R^n$ by setting $g(z)=T_x\circ f\circ
T^{-1}_x(z)$ for $z\in B^n$ and $g(z)=z$ for $z\in\mathbb
R^n\setminus B^n$. Then $g\in Id_K(\partial B^n)$with
$g(0)=T_x(f(x))$. By the invariance of $\rho_{B^n}$ under the group
${\mathcal{GM}}(B^n)$ of M\"obius selfautomorphisms of $B^n$ we see
that for $x\in B^n$
\begin{equation}
\label{rho_Bn}
\rho_{B^n}(f(x),x)=\rho_{B^n}(T_x(f(x)),T_x(x))=\rho_{B^n}(g(0),0).
\end{equation}
Choose $z\in\partial B^n$ such that $g(0)\in
[0,z]=\{tz\,:\,0\leqslant t\leqslant 1\}$. Let
$E'=\{-sz\,:\,s\geqslant 1\}$, $\Gamma'=\Delta([g(0),z],E';\mathbb
R^n)$ and $\Gamma=\Delta(g^{-1}[g(0),z],g^{-1}E';\mathbb R^n)$.

The spherical symmetrization with center at $0$ yields by \cite[Thm 8.44]{avv}
$$
M(\Gamma)\geqslant\tau_n(1)\quad(=2^{1-n}\gamma_n(\sqrt 2))
$$
because $g(x)=x$ for $x\in\mathbb R^n\setminus B^n$.
Next, we see by the choice of $\Gamma'$ that
$$
M(\Gamma')=\tau_n\left(\frac{1+|g(0)|}{1-|g(0)|}\right).
$$
By $K$-quasiconformality we have $M(\Gamma)\leqslant K\,M(\Gamma')$ implying
\begin{equation}
\label{main_inequality}
\exp(\rho_{B^n}(0,g(0)))=\frac{1+|g(0)|}{1-|g(0)|}\leqslant\tau_n^{-1}(\tau_n(1)/K)=\frac{1-a}a.
\end{equation}
The last equality follows from (\ref{aform}). 
Finally, (\ref{rho_Bn}) and (\ref{main_inequality}) complete the proof. $ \hfill \square$
\end{subsec}

\begin{subsec}{\bf
Proof of Theorem \ref{mycor}.} We have
$$
\begin{array}{rcl}
|f(x)-x| & \leqslant & \displaystyle
2\tanh\left(\frac{\rho_{B^n}(f(x),x)}{4}\right)
\leqslant
2\tanh\left(\frac{\log\left(\frac{1-a}{a}\right)}{4}\right)
\vspace{1em}\\
& \leqslant & \displaystyle
2\tanh\left(\frac{(K-1)(4+6\log 2)}{4}\right)
\vspace{1em}\\
& \leqslant & \displaystyle
(K-1)(2+3\log 2)\leqslant\frac 92(K-1).
\end{array}
$$
The first inequality follows from (\ref{tanh}), the second one from Theorem \ref{main_theorem},
the third one from Lemma \ref{stabrmk} and the last one from inequality
$\tanh(t)\leqslant t$ for $t\geqslant 0$.


For $n=2$ we use the same first two steps and planar case of Lemma \ref{stabrmk} to derive inequality
$$
|f(x)-x|\leqslant\frac{b}{2}(K-1). \quad \quad \hfill \square
$$
\end{subsec}

A lower bound corresponding to the upper bound in
(\ref{luckynumber}) is given in the next lemma.

\begin{lemma} For $f \in Id(\partial G)$ let
$$
\delta(f)  \equiv \sup \{|f(z)-z| : z \in G\}\,.
$$
Then for $f \in Id_K(\partial B^n), K>1,  \alpha=K^{1/(1-n)}$
\begin{equation} \label{**}
\delta(f) \ge (1-\alpha) \alpha^{\alpha/(1-\alpha)}
>\frac{1}{e}(1-\alpha).
\end{equation}
\end{lemma}

\begin{proof}
 The radial stretching $f: B^n \to B^n, n \ge 2, $ defined by
$f(z)=|z|^{\alpha-1}\,z, z \in B^n,$ ($0<\alpha<1$) is $K$-qc with
$\alpha=K^{1/(1-n)}$ \cite[p. 49]{v2} and $f \in Id_K(\partial
B^n)\,.$ Now we have
$$
|f(z)-z|=||z|^{\alpha-1}z-z|=|r^\alpha-r|,\quad |z|=r.
$$
Further, we see that
$$
\delta(f)=\sup_{0<r<1}(r^{\alpha}-r),
$$
where the supremum is attained for
$r=r_\alpha=\left(\frac{1}{\alpha}\right)^{\frac{1}{\alpha-1}}$, so
$$
\delta(f) =({1}-{\alpha}) \alpha^{\alpha/(1- \alpha)}\,.
$$
A crude, but simple, estimate is
$$
\delta(f)\ge (1/e)^{\alpha} - (1/e)
=\frac{1}{e}\left(\frac{1}{e^{\alpha
-1}}-1\right)=\frac{1}{e}\left(e^{1-\alpha}-1\right)
\geqslant\frac{1}{e}(1-\alpha)\, .
$$ 
\end{proof}

\begin{figure}[!ht]
\begin{center}
\includegraphics[width=100mm]{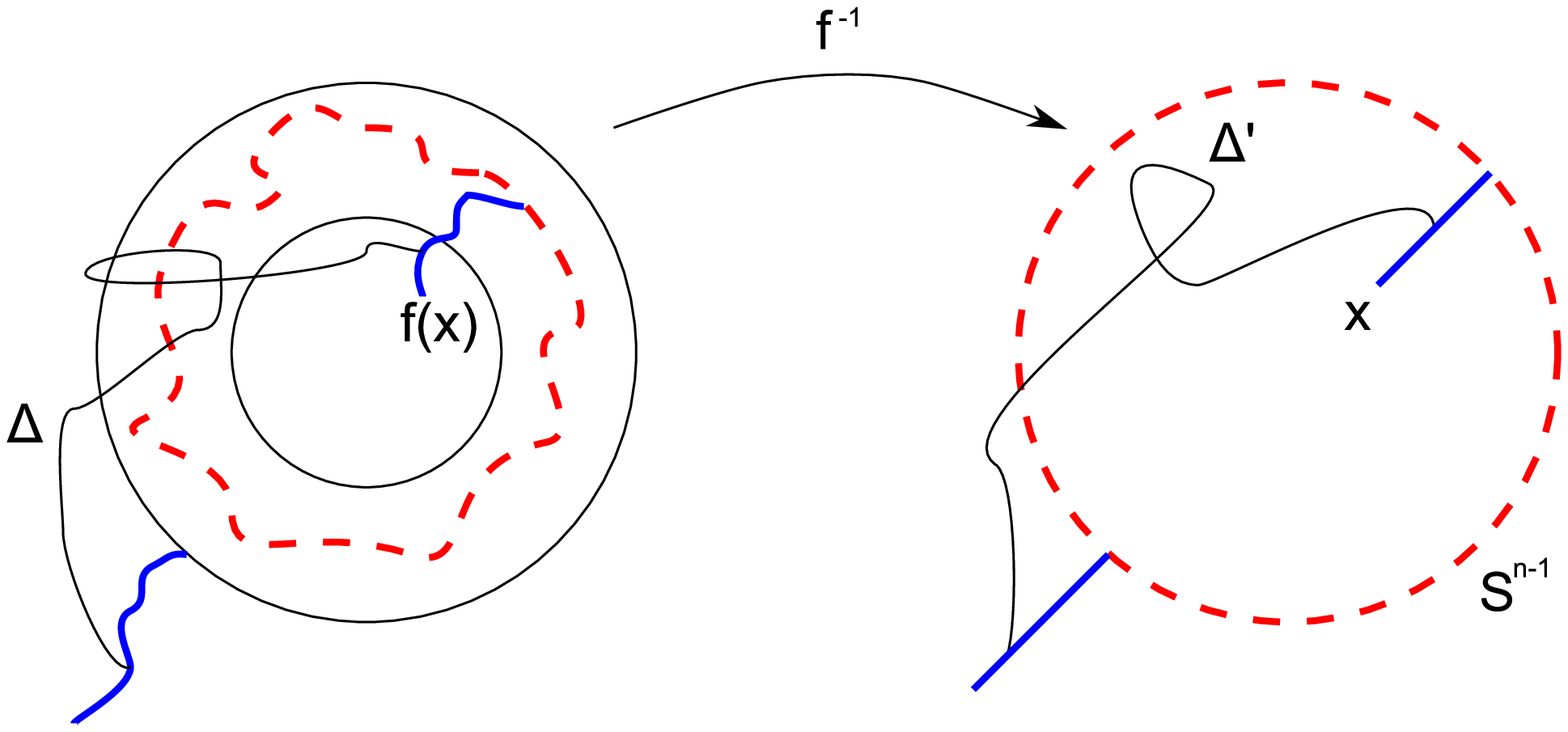}
\caption{}
\end{center}
\end{figure}

\begin{theorem} \label{etathm}
\label{bounds} Let $f:\overline{\mathbb
R^n}\longrightarrow\overline{\mathbb R^n}$ be a $K$-qc homeomorphism
with $f(\infty)=\infty$ and $B^n(m)\subset f(B^n) \subset B^n(M) $ where $0< m \le 1 \le M$. Then
$$
\eta_{1/K,n}\left(\frac{1+|x|}{1-|x|}\right)\leqslant\frac{M+|f(x)|}{m-|f(x)|}
$$
and
$$
\frac{m+|f(x)|}{M-|f(x)|}
\leqslant
\eta_{K,n}\left(\frac{1+|x|}{1-|x|}\right)
$$
for all $x\in B^n$ where $\eta_{K,n}(t)=\tau_n^{-1}(\tau_n(t)/K)$.

In particular, if $m=1=M$, then we have
$$
 \eta_{1/K,n}\left(\frac{1+|x|}{1-|x|}\right)\leqslant\frac{1+|f(x)|}{1-|f(x)|}
\leqslant
\eta_{K,n}\left(\frac{1+|x|}{1-|x|}\right)\, .$$
\end{theorem}

\begin{figure}[!ht]
\begin{center}
\includegraphics[width=100mm]{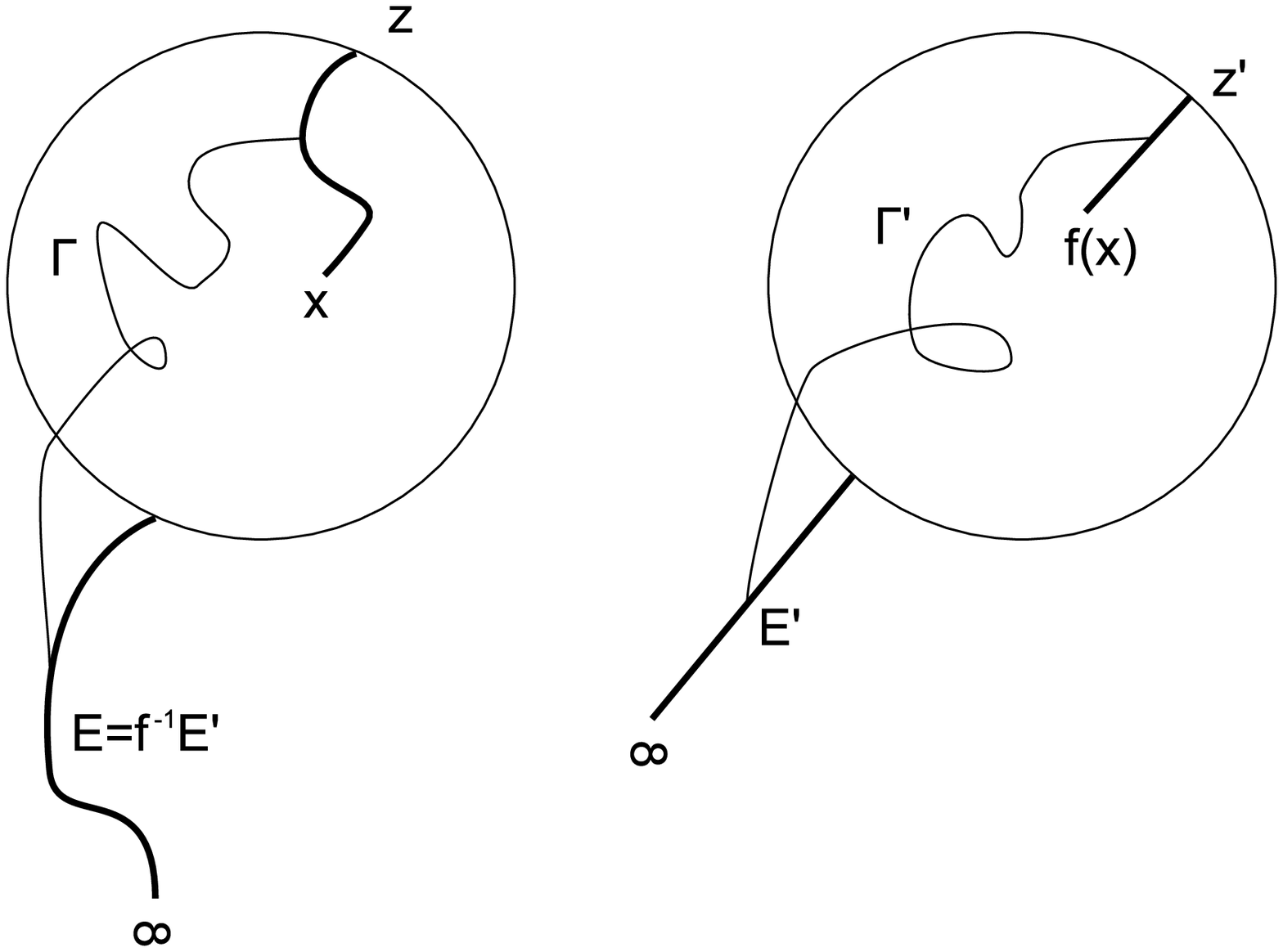}
\caption{}
\end{center}
\end{figure}

\begin{proof}
The proof is similar to the proof of Theorem \ref{main_theorem}. Fix
$x\in B^n$ and choose $z'\in\partial f(B^n)$ such that $f(x)\in[0,z']$ and $[f(x),z') \subset f(B^n)\,$
and fix $z" \in \partial f(B^n)$ such that $z',0,z"$ are on the same line, $0 \in[z', z"],$ and
$ \{-sz"\,:\,s\geqslant 1\} \subset {\mathbb R}^n \setminus f(B^n)\,.$
Let $\Gamma'=\Delta([f(x),z'],E';\mathbb{R}^n)$,
$E'=\{-sz"\,:\,s\geqslant 1\}$ and
$\Gamma=\Delta(f^{-1}[f(x),z'],f^{-1}E';\mathbb R^n)$. Then
$$
M(\Gamma')\le\tau_n\left(\frac{m+|f(x)|}{M-|f(x)|}\right)
$$
while applying a spherical symmetrization with center at the origin
gives
$$
M(\Gamma)\geqslant\tau_n\left(\frac{1+|x|}{1-|x|}\right)
$$
because $f^{-1}E'$ connects $\partial B^n$ and $\infty$.
Then the inequality $M(\Gamma)\leqslant K\,M(\Gamma')$
yields
$$
 \tau_n\left(\frac{1+|x|}{1-|x|}\right) \le K
 \tau_n\left(\frac{m+|f(x)|}{M-|f(x)|}\right),
$$
$$
\tau_n^{-1}( \frac{1}{K}\tau_n\left(\frac{1+|x|}{1-|x|}\right))
\ge\frac{m+|f(x)|}{M-|f(x)|}
$$

\begin{equation}
\label{eta}
\frac{m+|f(x)|}{M-|f(x)|}\leqslant\eta_{K,n}\left(\frac{1+|x|}{1-|x|}\right).
\end{equation}
The lower bound follows if we apply a similar argument to $f^{-1}$ and the lower bound
$$
M(\Gamma')\ge\tau_n\left(\frac{M+|f(x)|}{m-|f(x)|}\right) \,.
$$
\end{proof}

\begin{subsec}{\bf Remark.} \label{mynewobs}
Putting $x=0, m=1=M$ in (\ref{eta}) we obtain by (\ref{aform})
for  a $K$-qc homeomorphism $f:\overline{\mathbb
R^n}\longrightarrow\overline{\mathbb R^n}$
with $f(\infty)=\infty$ and $ f(B^n) = B^n $  that
$$   |f(0)| \le 1- 2a\,, a= \varphi_{1/K,n}(1/\sqrt{2})^2\,.$$
Further, if we use the lower bound (\ref{7502}) from Lemma
\ref{cgqm747} we obtain
$$   |f(0)| \le 1- 2^{1- \beta} 4^{1-K} K^{-2K}\,. $$
In the special case when $n=2$ we have
$$|f(0)| \le 1- 2^{3(1- K)}  K^{-2K}\le (2+ 3 \log 2)(K-1)\,.  $$
Note that this last inequality does not suppose that $f \in Id_K(\partial B^n)\,,$
only the hypotheses of Theorem \ref{etathm} are needed.
\end{subsec}

\begin{subsec}{Maps of cylinder}
{\rm We next consider the class $Id_K(\partial Z)$ for the case when the
domain $Z$ is an infinite cylinder.}
\end{subsec}
\begin{theorem}
Let $Z=\{(x,t)\in\mathbb R^n\,:\,|x|<1,\,t\in\mathbb R\}$, $f\in Id_K(\partial Z)$. Then
$k_Z(0,f(0))\leqslant c(K)$ where $c(K)\rightarrow 0$ when $K\rightarrow 1$.
\end{theorem}

\begin{figure}[!ht]
\begin{center}
\includegraphics[width=50mm]{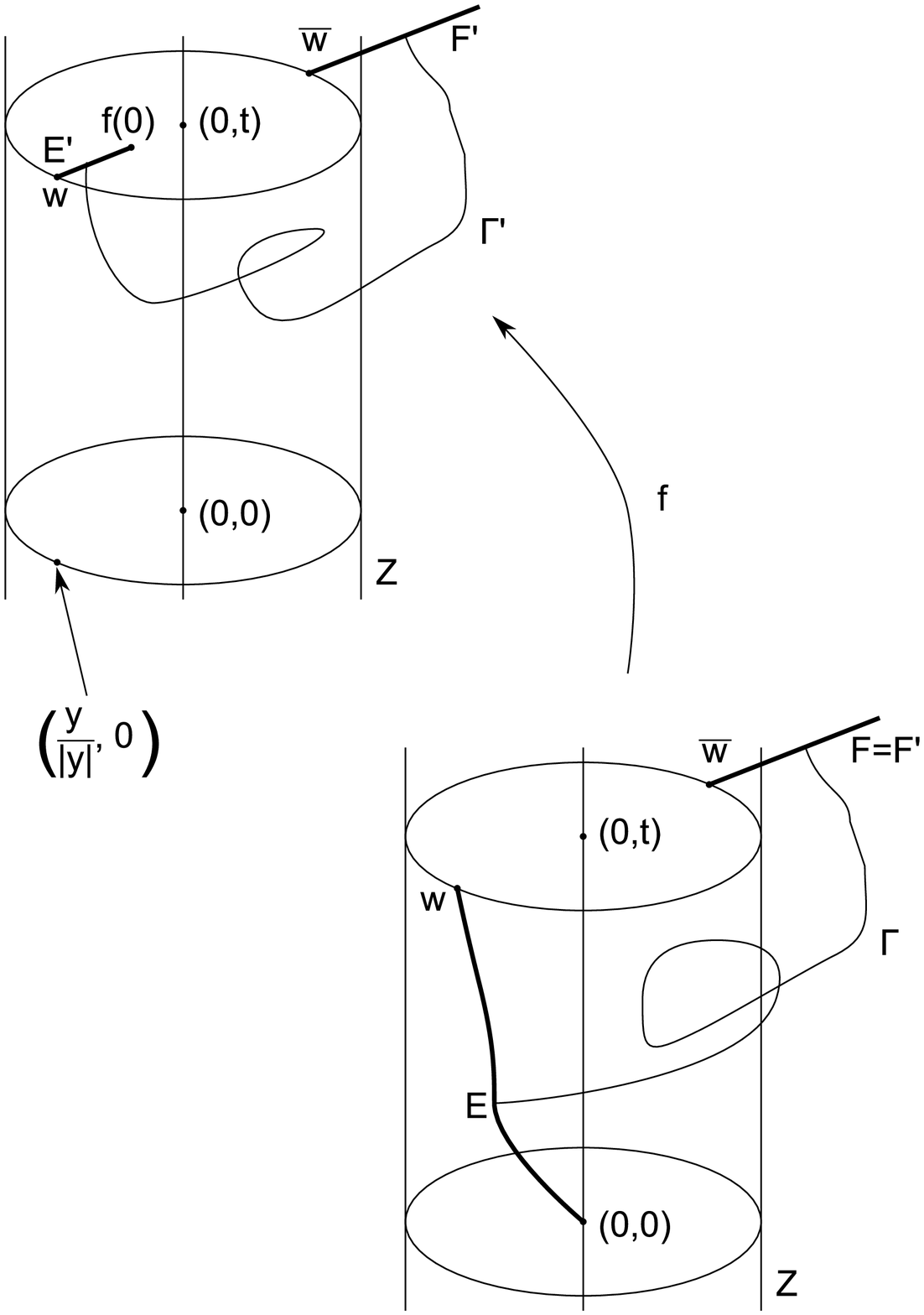}
\caption{}
\end{center}
\end{figure}

\begin{proof}
Let $f(0)=(y,t)$, $E'=[w,f(0)]$,
$F'=\{\overline w+s(y,0)\,:\,s\leqslant 0\}$ where
$w=(y/|y|,t)$, $\overline w=(-y/|y|,t)$. Then $E'$ and $F'$ are the complementary components of a
Teichm\" uller ring and therefore writing
$\Gamma'=\Delta(E',F';\mathbb R^n)$ we have
$$
M(\Gamma')\leqslant\tau_n\left(\frac{1+|y|}{1-|y|}\right).
$$
The modulus of the family $\Gamma=\Delta(E,F;\mathbb R^n)$, $E=f^{-1}E'$, $F=f^{-1}F'$ can be
estimated by use of spherical symmetrization with the center at 0. Note that $E=E'$
because $E'\subset\mathbb R^n\setminus Z$ and $f\in Id_K(\partial Z)$. By \cite[7.34]{vu2}
we have
$$
M(\Gamma)\geqslant\tau_n(1).
$$
By $K$-quasiconformality $M(\Gamma)\leqslant K\,M(\Gamma')$ implying
$$
\exp(\rho_{B^{n-1}}(0,y))=\frac{1+|y|}{1-|y|}\leqslant\tau_n^{-1}\left(\frac{\tau_n(1)}{K}\right).
$$

\begin{figure}[!ht]
\begin{center}
\includegraphics[width=70mm]{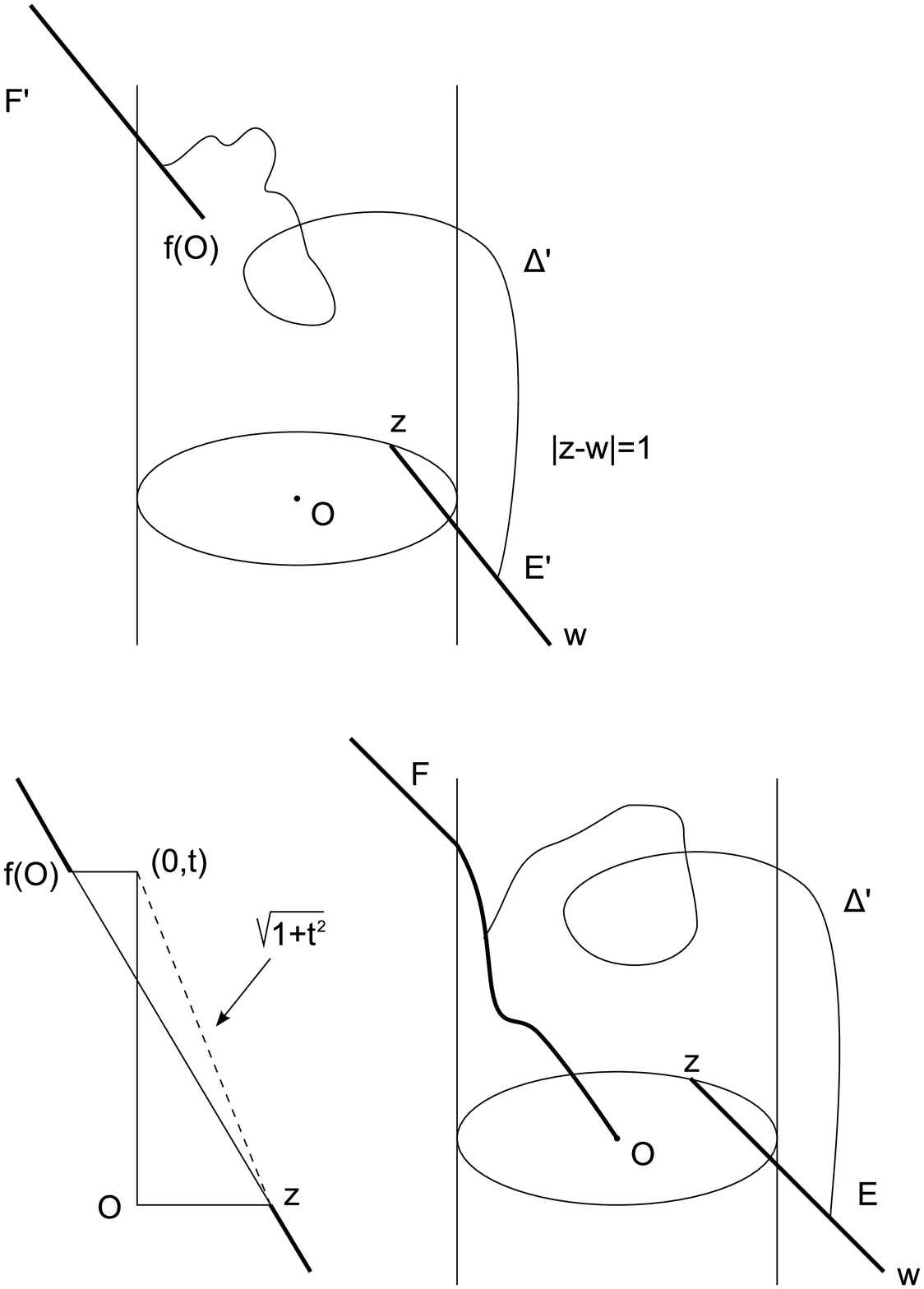}
\caption{}
\end{center}
\end{figure}

Next we shall estimate $t$. Fix first $z$ in $\{w\in\partial Z\,:\,w_n=0\}$ such that $|f(0)-z|$
is maximal. Then choose a point $w$ on the line through $f(0)$ and $z$ such that $|z-w|=1$
and $[z,w]\subset\mathbb R^n\setminus Z$. Let $E'=[z,w]$ and $F'=\{f(0)+t(f(0)-z)\,:\,t\geqslant 0\}$.
Then $E'$ and $F'$ are the complementary components of a Teichm\" uller ring and with
$\Delta'=\Delta(E',F';\mathbb R^n)$ we have
$$
M(\Delta')=\tau_n(|f(0)-z|).
$$
Observing that $E'=f^{-1}E'$, because $f\in Id_K(\partial Z)$ and
carrying out a spherical symmetrization with center at $z$ we see
that if $E=f^{-1}E'$, $F=f^{-1}F'$ then
$$
M(\Delta)\geqslant\tau_n(1),\quad\Delta=\Delta(E,F;\mathbb R^n).
$$
By $K$-quasiconformality we have
$$
1+t^2\leqslant|f(0)-z|^2\leqslant\tau_n^{-1}\left(\frac{\tau_n(1)}{K}\right)^2.
$$
The triangle inequality for $k_Z$ yields
$$
\begin{array}{rcl}
k_Z(0,f(0)) & \leqslant & k_Z(0,(0,t))+k_Z((0,t),(y,t))
\vspace{0.5em}\\
& = & t+k_{B^{n-1}}(0,y)\leqslant|t|+2\,\rho_{B^{n-1}}(0,y)
\vspace{0.5em}\\
& \leqslant &
\sqrt{\tau_n^{-1}\left(\frac{\tau_n(1)}{K}\right)^2-1}
+2\log\left(\tau_n^{-1}\left(\frac{\tau_n(1)}{K}\right)\right)
\vspace{0.5em}\\
& \leqslant &
\displaystyle
\sqrt{e^{18(K-1)}-1}+18(K-1).
\end{array}
$$
The last inequality follows from (\ref{aform}) and Lemma \ref{stabrmk}.
\end{proof}

\section{Distortion of two point normalized quasiconformal mappings}

Let $\eta \colon [0,\infty) \to [0,\infty)$ be an increasing homeomorphism and $D,D' \subset \Rn$. A homeomorphism $f \colon D \to D'$ is \emph{$\eta$-quasisymmetric} if
\begin{equation}\label{qsdefinition}
  \frac{|f(a)-f(c)|}{|f(b)-f(c)|} \le \eta \left( \frac{|a-c|}{|b-c|} \right)
\end{equation}
for all $a,b,c \in D$ and $c \ne b$. By \cite{v2} K-quasiconformal mapping of the whole $\Rn$ is $\eta_{K,n}$- quasisymmetric with a control function $\eta_{K,n}$. Let us define the optimal control function by
\[
  \eta^*_{K,n}(t) = \sup \{ |f(x)| \colon |x| \le t, f \in QC_K(\Rn), f(y) = y \textnormal{ for } y \in \{ 0,e_1,\infty \} \}.
\]

Vuorinen \cite[Theorem 1.8]{vu2} proved an upper bound for $\eta^*_{K,n}(t)$, which was later refined by Prause \cite[Theorem 2.7]{p} for $K < 4/3$ into the following form
\begin{equation}
  \eta^*_{K,n}(t) \le \left\{ \begin{array}{ll}
    \displaystyle \eta^*_{K,n}(1)\p_{K,n}(t), & 0 < t < 1,\\
    \displaystyle 1+600 \left( (K-1)\log\frac{1}{K-1} \right), & t = 1,\\
    \displaystyle \eta^*_{K,n}(1)\frac{1}{\p_{1/K,n}(1/t)}, & t > 1,
  \end{array} \right.
\end{equation}
where
\begin{equation}\label{etastar1}
    \eta^*_{K,n}(1) \le \exp((4\sqrt{2}-\log(K-1))(K^2-1)).
\end{equation}
We also introduce a simpler estimate of $\eta^*_{K,n}(1)$ from \cite[Theorem 14.8]{avv}
\begin{equation}\label{etastar1simpler}
    \eta^*_{K,n}(1) \le \exp(4K(K+1)\sqrt{K-1}).
\end{equation}

A more rough upper bound for $\eta^*_{K,n}(t)$ can \cite[Theorem 7.47]{vu1} be written as
\begin{equation}\label{roughupperbound}
  \eta^*_{K,n}(t) \le \left\{ \begin{array}{ll}
    \displaystyle \eta^*_{K,n}(1)\lambda_n^{1-\alpha}t^\alpha, & 0 < t \le 1,\\
    \displaystyle \eta^*_{K,n}(1)\lambda_n^{1-\beta}t^{\beta}, & t > 1,
  \end{array} \right.
\end{equation}
where $\alpha = K^{1/(1-n)}$ and $\beta = 1/\alpha$. Furthermore, we can \cite[Lemma 7.50]{vu1} estimate
\begin{equation}\label{upperboundsoflambda}
  \lambda_n^{1-\alpha} \le 2^{1-1/K}K \quad \textnormal{and} \quad \lambda_n^{1-\beta} \le 2^{1-K}K^{-K}.
\end{equation}

\begin{lemma} \label{mypropo}
  Let $K \in (1,2]$, $f \in QC_K(\Rn)$, $f(x)=x$ for $x \in \{ 0,e_1 \}$, $\alpha = K^{1/(1-n)}$ and $\beta = 1/\alpha$. Then
  \[
    \begin{array}{ll}
      \displaystyle \frac{1}{c_3}|x|^\beta \le |f(x)| \le c_3 |x|^\alpha, & \textnormal{if } 0<|x|\le1,\\
      \displaystyle \frac{1}{c_3}|x|^\alpha \le |f(x)| \le c_3 |x|^\beta, & \textnormal{if } |x|>1
    \end{array}
  \]
  for $c_3 = \exp (60 \sqrt{K-1})$.
\end{lemma}
\begin{proof}
  Since $f$ is quasiconformal it is also $\eta^*_{K,n}$-quasisymmetric and by choosing $a=x$, $b=0$ and $c=e_1$ in (\ref{qsdefinition}) we have $|f(x)| \le \eta^*_{K,n}(|x|)$. Similarly, selection $(a,b,c) = (e_1,0,x)$ in (\ref{qsdefinition}) gives $|f(x)| \ge 1/\eta^*_{K,n}(1/|x|)$. Therefore
  \begin{equation}\label{etabounds}
    \frac{1}{\eta^*_{K,n}(1/|x|)} \le |f(x)| \le \eta^*_{K,n}(|x|)
  \end{equation}
  for all $x \in \overline{\R}^n \setminus \{ 0 \}$. Therefore by (\ref{roughupperbound})
  \[\begin{array}{ll}
    \displaystyle \frac{1}{c_2}|x|^\beta \le |f(x)| \le c_1 |x|^\alpha, & \textnormal{if } 0<|x|<1,\\
    \displaystyle \frac{1}{\eta^*_{K,n}(1)} \le |f(x)| \le \eta^*_{K,n}(1), & \textnormal{if } |x| = 1,\\
    \displaystyle \frac{1}{c_1}|x|^\alpha \le |f(x)| \le c_2 |x|^\beta, & \textnormal{if } |x|>1,\\
  \end{array}\]
  for $c_1 = \eta^*_{K,n}(1)\lambda_n^{1-\alpha}$ and $c_2= \eta^*_{K,n}(1)\lambda_n^{1-\beta}$. We can estimate $\max \{ c_1,c_2 \} \le c_3 = \exp (60 \sqrt{K-1})$ for $K \in (1,2]$.
\end{proof}

We will consider $K$- quasiconformal mapping $f \colon \overline{\R}^n \to \overline{\R}^n$ with $f(y) = y$ for $y \in \{ 0,e_1,\infty \}$ and our goal is to find an upper bound for $|f(x)-x|$ or similar quantities in terms of $K$ and $n$, when $|x| \le 2$ and $K > 1$ is small enough.

Fix $x \in \Rn \setminus \{ 0,e_1 \}$ and assume that $|x|-\e \le |f(x)| \le |x|+\e$ and $|x-e_1|-\e \le |f(x)-e_1| \le |x-e_1|+\e$ for $\e \in (0,\min \{ |x|,|x-e_1| \})$. Now
\begin{equation}\label{diamestimate}
  |f(x)-x| \le \frac{\diam(A)}{2},
\end{equation}
where
\[
  A = A(0,|x|+\e,|x|-\e) \cap A(e_1,|x-e_1|+\e,|x-e_1|-\e) \cap \{ z \in R^3 \colon z_3=0 \}
\]
and
\[
  A(z,R,r) = B^n(z,R) \setminus \overline{B}^n(z,r).
\]
We will now find upper bounds for $\diam (A)$.
\begin{figure}[!ht]
  \begin{center}
    \includegraphics[width=8cm]{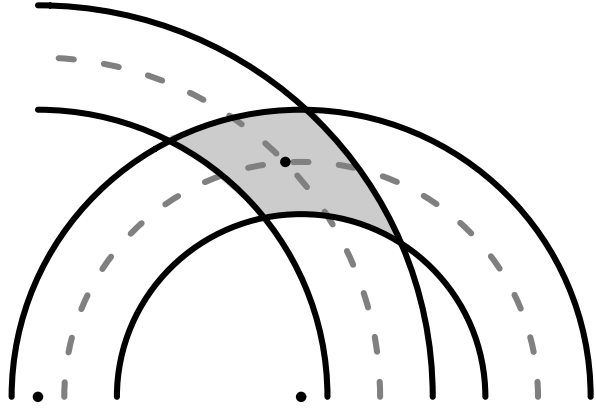}
    \caption{The set $A$.}
  \end{center}
\end{figure}

\begin{theorem}\label{bound3}
  For $\e < 1$ and $A$ and $x$ as in (\ref{diamestimate})
  \[
    \diam(A) \le \sqrt{\e}4(\min \{ |x|,|x-e_1| \} +1).
  \]
\end{theorem}
\begin{proof}
  Let us assume $|x| \le |x-e_1|$. Now $\diam (A)$ is maximal when $|x-e_1|/|x|$ is maximal. Therefore we may assume $x = -s$, where $s > 0$. Denote ${y} = S^1(0,|x|+\e) \cap S^1(e_1,|x-e_1|-\e)$. Area of the triangle $\triangle e_10y$ is $(\im y)/2$ and by Heron's formula
  \begin{equation}\label{heron}
    \frac{\im y}{2} = \sqrt{p(p-1)(p-|x|-\e)(p+\e-|x|-1)},
  \end{equation}
  where $p = |x|+1$. By (\ref{heron}) and the assumption $\e < 1$ we have
  \[
    \im y = 2\sqrt{(|x|+1)|x|(1-\e)\e} \le 2\sqrt{\e}(|x|+1)
  \]
  and the assertion follows since $\diam (A) \le 2 \im y$.
\end{proof}

\begin{theorem}\label{bound4}
  Let $A$ be as in (\ref{diamestimate}), $|x| < 2$, $|x-e_1| \le |x|$ and $\measuredangle (1,0,x) \ge \omega > 0$. Then
  \[
    \diam(A) \le \e \left( 1+\frac{70}{\omega} \right)   \]
  for
  \[
    \e < \min \left\{1, \frac{1+|x-e_1|-|x|}{2},\frac{|x|+|x-e_1|-1}{2} \right\}.
  \]
\end{theorem}
\begin{proof}
  Let us denote by $y$ the intersection of $S^1(|x|+\e)$ and $S^1(e_1,|x-e_1|)$ in the first quadrant. The triangles $\triangle (0,1,x)$ and $\triangle (0,1,y)$ give by the Law of Cosines
  \[
    |x-1|^2 = |x|^2+1-2 |x| \cos \gamma
  \]
  and
  \[
    |y-1|^2 = |y|^2+1-2 |y| \cos \delta,
  \]
  where $\eta$ is the angle $\measuredangle (1,0,x)$ and $\xi$ is the angle $\measuredangle (1,0,y)$. Therefore
  \begin{equation}\label{cosgamma}
    \cos \gamma = \frac{|x|^2+1-|x-1|^2}{2 |x|}
  \end{equation}
  and
  \begin{equation}\label{cosdelta}
    \cos \delta = \frac{|y|^2+1-|y-1|^2}{2 |y|} = \frac{(|x|+\e)^2+1-|x-1|^2}{2 (|x|+\e)}.
  \end{equation}
  By the Jordan inequality
  \[
    |\cos \gamma-\cos \delta| = \cos \delta-\cos \gamma = 2\sin\frac{\delta+\gamma}{2}\sin\frac{\gamma-\delta}{2} \ge \frac{2}{\pi^2}(\gamma+\delta)(\gamma-\delta)
  \]
  and by assumption
  \begin{equation}\label{gamma-delta}
    |\gamma-\delta| \le \frac{\pi^2}{2 \omega} |\cos \gamma-\cos \delta|.
  \end{equation}
  By the triangle inequality, the Jordan inequality, (\ref{gamma-delta}), (\ref{cosgamma}) and (\ref{cosdelta})
  \begin{eqnarray*}
    |x-y| & \le & \e+2|x|\sin \frac{|\gamma-\delta|}{2}\\
    & \le & \e+\frac{2|x|}{\pi}|\gamma-\delta|\\
    & \le & \e+\frac{|x|\pi}{\omega} |\cos \gamma-\cos \delta|\\
    & = & \e+\frac{|x|\pi}{\omega} \frac{\e (1+|x-1|^2+|x|(|x|+\e))}{2|x|(|x|+\e)}\\
    & \le & \e+\frac{\pi}{\omega} \frac{\e(1+3^2+2(2+1))}{2(1/2+0)}\\
    & = & \e+\frac{51 \e}{\omega}.
  \end{eqnarray*}

  Let us denote by $z$ the intersection of $S^1(|x|+\e)$ and $S^1(e_1,|x-e_1|+\e)$ in the first quadrant. If  $\delta > \omega/2$, then we obtain
  \[
    |x-y| \le \frac{|x|\pi}{\omega} \frac{\e ||y-1|^2-|z-1|^2|}{2(|x|+\e)} \le \frac{19 \e}{\omega}.
  \]
  Now
  \[
    \diam(A) \le |x-z| \le |x-y|+|y-z| \le \e+\frac{70\e}{\omega}
  \]
  and the assertion follows.
\end{proof}

\begin{lemma}\label{c3estimate}
  Let $n \ge2$, $K>1$, $\alpha = K^{1/(1-n)}$, $\beta=1/\alpha$ and $c_3 = \exp (60 \sqrt{K-1})$. For $t \in (0,1)$
  \begin{equation}\label{c3estimate1}
    c_3 t^\alpha -t \ge t-\frac{t^\beta}{c_3}
  \end{equation}
  and for $t > 1$
  \begin{equation}\label{c3estimate2}
    c_3 t^\beta -t \ge t-\frac{t^\alpha}{c_3}.
  \end{equation}
\end{lemma}
\begin{proof}
  To prove (\ref{c3estimate1}) it is sufficient to prove that $f(t)\ge 0$, where $f(t)=c_3t^{\alpha}+\frac{1}{c_3}t^{1/\alpha}-2t$ and $0<t<1$. Because
  $$
    \lim_{t\rightarrow 0+}f(t)=0,
  $$
  it is sufficient to prove $f'(t)\ge 0$ for $0<t<1$, i.e.
  \begin{equation}\label{first_derive}
    \alpha c_3t^{\alpha-1}+\frac{1}{\alpha c_3}t^{(1/\alpha)-1}-2\ge 0.
  \end{equation}
  Using inequality between arithmetic and geometric means, we can conclude that
  $$
    \alpha c_3+\frac{1}{\alpha c_3}\ge 2
  $$
  holds. In other words,
  $$
    \lim_{t\rightarrow 1-}f'(t)=\alpha c_3+\frac{1}{\alpha c_3}-2\ge 0.
  $$
  By this reason, to prove inequality (\ref{first_derive}) it is sufficient to prove that $f''(t)\le 0$ for $0<t<1$ i.e.
  $$
    \alpha(\alpha-1)c_3t^{\alpha-2}+\frac{\frac{1}{\alpha}-1}{\alpha c_3}t^{(1/\alpha)-2}\le 0,
  $$
  or equivalently
  $$
    t^{\frac{1}{\alpha}-\alpha}\le \alpha^3c_3^2.
  $$
  The last inequality follows from
  $$
    t^{\frac{1}{\alpha}-\alpha}<1\le \alpha^3c_3^2.
  $$
  The first inequality holds because $0<t<1$ and $\frac 1\alpha-\alpha>0$ (because $0<\alpha<1$). Now we prove $\alpha^3c_3^2\ge 1$ to complete proof. This is same inequality as
  $$
    K^{3/(1-n)}e^{120\sqrt{K-1}}\ge 1,
  $$
  or equivalently
  $$
    e^{40(n-1)u}\ge u^2+1
  $$
  for $u = \sqrt{K-1}$. Because $u\ge 0$, using Taylor series for $e^x$ we can conclude that
  $$
    e^{40(n-1)u}\ge 1+40(n-1)u+\frac{(40(n-1))^2u^2}{2}\ge 1+u^2.
  $$

  The inequality (\ref{c3estimate2}) is equivalent to
  \begin{equation}
  \label{ac3}
    c_3t^{(1/\alpha)-1}+\frac{t^{\alpha-1}}{c_3}\ge 2.
  \end{equation}
  Inequality (\ref{ac3}) holds for $t=1$. To prove inequality (\ref{ac3}) for $t>1$ it is sufficient to prove that derivation of the left side of inequality is nonnegative. We have following sequence of equivalent formulas:
  $$
    c_3\left(\frac{1}{\alpha}-1\right)t^{\frac{1}{\alpha}-2}+\frac{\alpha-1}{c_3}\,t^{\alpha-2}\ge 0
  $$
  $$
    \frac{(1-\alpha)c_3}{\alpha}\,t^{\frac{1}{\alpha}-2}\ge \frac{1-\alpha}{c_3}\,t^{\alpha-2}
  $$
  $$
    t^{\frac{1}{\alpha}-\alpha}\ge\frac{\alpha}{c_3^2}.
  $$
  The last inequality is true because
  $$
    t^{\frac{1}{\alpha}-\alpha}\ge 1\ge\frac{\alpha}{c_3^2}
  $$
  and the assertion follows.
\end{proof}

\begin{lemma}\label{epsilonestimate}
  Let $\e > 0$. Then
  \[
    |x|-\e \le |f(x)| \le |x|+\e
  \]
  for
  \[
    1 < K \le \max \left\{ \left( \frac{\log (\e+1)}{60} \right)^2+1 ,2 \right\}.
  \]
\end{lemma}
\begin{proof}
  Let us denote $l(x) = c_3^{-1} \max \{ |x|^\alpha,|x|^\beta \}$ and $u(x) = c_3 \max \{ |x|^\alpha,|x|^\beta \}$.

  We will first consider the case $0 < |x| < 1$. By Lemma \ref{c3estimate}
  \begin{eqnarray*}
    \max \{ u(x)-|x|,|x|-l(x) \} & = & \max \left\{ c_3 |x|^\alpha-|x|,|x|-\frac{1}{c_3}|x|^\beta \right\}\\
     & = & c_3 |x|^\alpha-|x|\\
    & \le & \exp(60\sqrt{K-1}) |x|^\alpha - |x|\\
    & \le & \exp(60\sqrt{K-1}) |x|^{1/K} - |x|\\
    & \le & \exp(60\sqrt{K-1})-1.
  \end{eqnarray*}
  Now $\exp(60\sqrt{K-1})-1 \le \e$ is equivalent to
  \begin{equation}\label{Kupperbound1}
    K \le \left( \frac{\log (\e+1)}{60} \right)^2+1.
  \end{equation}

  If $|x| = 1$, then
  \[
    \max \{ u(x)-|x|,|x|-l(x) \} = c_3-1
  \]
  and therefore we want $\exp(60\sqrt{K-1})-1 \le \e$ for $K \in (1,2]$, which is equivalent to
  \begin{equation}\label{Kupperbound2}
    K \le \left( \frac{\log (\e+1)}{60} \right)^2+1.
  \end{equation}

  Let us first consider the case $1 < |x| < 2$. By Lemma \ref{c3estimate}
  \begin{eqnarray*}
    \max \{ u(x)-|x|,|x|-l(x) \} & = & \max \left\{ c_3 |x|^\beta-|x|,|x|-\frac{1}{c_3}|x|^\alpha \right\}\\
    & = & c_3 |x|^\beta-|x|\\
    & \le & \exp(60\sqrt{K-1}) |x|^\beta - |x|\\
    & \le & \exp(60\sqrt{K-1}) |x|^K - |x|\\
    & \le & |x|(\exp(60\sqrt{K-1})|x|^{K-1}-1)\\
    & \le & 2(\exp(60\sqrt{K-1}+(K-1)\log |x|)-1)\\
    & \le & 2(\exp(60(K-1)^{3/2}-1).
  \end{eqnarray*}
  Now $2(\exp(60(K-1)^{3/2}-1) \le \e$ is equivalent to
  \begin{equation}\label{Kupperbound3}
    K \le \left( \frac{\log (\e/2+1)}{60} \right)^{2/3}+1.
  \end{equation}
  By combining (\ref{Kupperbound1}), (\ref{Kupperbound2}) and (\ref{Kupperbound3}) we have
  \[
  |x|-\e \le |f(x)| \le |x|+\e
  \]
  for
  \[
    K \le \min \left\{ \left( \frac{\log (\e+1)}{60} \right)^2+1,2,\left( \frac{\log (\e/2+1)}{60} \right)^{2/3}+1 \right\} = \left( \frac{\log (\e+1)}{60} \right)^2+1
  \]
  and the assertion follows.
\end{proof}

\comment{
\begin{theorem}
  Let $\e > 0$, $f \in QC_K(\R^3)$, $f(z) = z$ for $z \in \{ 0,e_1 \}$ and
  \[
    1 < K \le \max \left\{ \left( \frac{\log (\e+1)}{60} \right)^2+1 ,2 \right\}.
  \]
  Then
  \[
    |p(f(x))-p(x)| \le 12 \sqrt{\e}
  \]
  for $x \in B^3(2)$ and $p(x) = p((x_1,x_2,x_3)) = (x_1,\sqrt{x_2^2+x_3^2},0)$.
\end{theorem}
\begin{proof}
  The assertion follows by Lemma \ref{epsilonestimate} and Theorem \ref{bound3}.
\end{proof}
}

\begin{lemma}\label{logestimate}
  For $0 < \alpha < 1$, $c > 1$ and $t > 0$
  \[
    \log(1+c \max \{ t^\alpha,t^{1/\alpha} \}) \le \left\{ \begin{array}{ll} \frac{c}{\alpha} \log^\alpha (1+t), & 0<t<1,\\ \frac{c}{\alpha} \log (1+t), & t \ge 1. \end{array} \right.
  \]
\end{lemma}
\begin{proof}
  Let us first assume $t \ge 1$. Then $\max \{ t^\alpha,t^{1/\alpha} \} = t^{1/\alpha}$ and by the generalized Bernoulli inequality
  \[
    (1+t)^{c/\alpha} \ge (1+ct)^{1/\alpha} \ge 1+ c^{1/\alpha} t^{1/\alpha} \ge 1+ c t^{1/\alpha}
  \]
  implying $\log(1+c t^{1/\alpha}) \le c/\alpha \log (1+t)$.

  Let us then assume $0 < t < 1$. Now $\max \{ t^\alpha,t^{1/\alpha} \} = t^\alpha$, we will show that function
  \[
    f(t) = \log (1+c t^\alpha)-\frac{c}{\alpha}\log^\alpha (1+t)
  \]
  is nonpositive. We easily obtain
  \begin{equation}\label{derivativeoff}
    f'(t) = \frac{\alpha c t^{\alpha-1}}{1+ct^\alpha} -\frac{c \log^{\alpha-1}(1+t)}{1+t}.
  \end{equation}
  Since $\alpha-1 < 0$ and $\log(1+t) \le t$ we have
  \begin{equation}\label{tandlogt}
    t^{\alpha-1} \le \log^{\alpha-1}(1+t).
  \end{equation}
  By assumptions $c > 1 \ge t^{1-\alpha}$ and therefore
  \begin{equation}\label{ctandt}
    c t^\alpha \ge t.
  \end{equation}
  By (\ref{derivativeoff}), (\ref{tandlogt}) and (\ref{ctandt}) $f'(t) \le 0$ is equivalent to $(1-\alpha)(1+t) \ge 0$ and therefore $f(t)$ is increasing. Now we have $f(t) \le f(0) = 0$ and the assertion follows.
\end{proof}

\begin{theorem}
  Let $G = \Rn \setminus \{ 0 \}$, $f \in QC_K$ and $f(0) = 0$. There exists $c(K)$ such that
  \[
    j_G(f(x),f(y)) \le c(K) \max \{ j_G(x,y)^\alpha , j_G(x,y) \},
  \]
  where $\alpha = K^{1/(1-n)}$, and $c(K) \to 1$ as $K \to 1$.
\end{theorem}
\begin{proof}
  By symmetry we may assume $x = e_1$ and $|y| \ge 1$. Now
  \[
    \frac{|f(y)-f(e_1)|}{|e_1|} = \frac{|f(y)-f(x)|}{|f(0)-f(e_1)|} \le \eta \left( \frac{|x-y|}{|0-e_1|} \right) = \eta (|x-y|)
  \]
  and
  \[
    \frac{|f(y)-f(e_1)|}{|f(y)|} = \frac{|f(y)-f(x)|}{|f(y)-f(0)|} \le \eta \left( \frac{|x-y|}{|y-0|} \right) = \eta \left( \frac{|x-y|}{|y|} \right).
  \]
  Therefore by Lemma \ref{logestimate}
  \begin{eqnarray*}
    j(f(x),f(y)) & = & \log \left( 1+\frac{|f(x)-f(y)|}{\min \{ |f(x)|,|f(y)| \}} \right)\\
    & = & \log \left( 1+\max \left\{ \eta(|y-e_1|),\eta \left( \frac{|x-y|}{|y|} \right) \right\} \right)\\
    & = & \log (1+\eta(|y-e_1|))\\
    & \le & \log (1+c_3 \max \{ |y-e_1|^\alpha,|y-e_1|^{1/\alpha} \})\\
    & \le & \left\{ \begin{array}{ll} \frac{c_3}{\alpha} \log^\alpha (1+|y-e_1|), & 0<|y-e_1|<1,\\ \frac{c_3}{\alpha} \log (1+|y-e_1|), & |y-e_1| \ge 1. \end{array} \right.
  \end{eqnarray*}
  By choosing $c(K) = c_3/\alpha$ we have $c(K) \to 1$ as $K \to 1$ by
  Lemma \ref{mypropo} and the assertion follows.
\end{proof}

\begin{problem}\rm
Is this Theorem true for $k_G$ instead of $j_G$?
\end{problem}

\chapter{Harmonic Quasiregular Mappings}

It is well known that if $f$ is a complex-valued harmonic function defined in a region $G$ of the
complex plane $\mathbb C,$ then $|f|^p$ is subharmonic for $p \ge 1,$ and that in the general
case is not subharmonic for $p < 1.$ However, if $f$ is holomorphic, then $|f|^p$ is
subharmonic for every $p > 0.$ Here we consider $k$-quasiregular harmonic
functions $(0 < k < 1).$ We recall that a harmonic function is quasiregular if
$$  |\bar\partial f (z)|\le k|\partial f (z)|, \qquad z\in G,$$
where
$$ \bar\partial f(z)= \frac 12\left(\frac{\partial f}{\pa x}+i\frac{\partial f}{\pa y}\right)\quad \text{and} \quad \partial f(z)= \frac 12\left(\frac{\partial f}{\pa x}-i\frac{\partial f}{\pa y}\right), \qquad z=x+iy.$$
We prove that $|f|^p$ is subharmonic for $p \ge  4k/(1 + k)^2 =: q$ as well as that the
exponent $q$ $(< 1)$ is the best possible (see Theorem \ref{th-1}). The fact that $q < 1$ enables
us to prove that if $f$ is quasiregular in the unit disk $\mathbb D$ and continuous on $\overline{D\mathstrut}$, then
$\tilde\omega(f, \delta) \le{\rm const.}\omega(f, \delta),$ where $\tilde\omega(f, \delta)$
(respectively $\omega(f, \delta)$) denotes the modulus of
continuity of $f$ on $\mathbb D$ (respectively $\pa\mathbb D$); see Theorem \ref{th-2}.

\section{ Subharmonicity of $|f|^p$}
\begin{theorem}\cite{kp}\label{th-1} If $f$ is a complex-valued $k$-quasiregular harmonic
function defined on a region $G \subset \mathbb C,$
and $q = 4k/(k + 1)^2,$ then $|f|^q$ is subharmonic. The exponent $q$ is optimal.
\end{theorem}

Recall that a  continuous function $u $ defined on a region $G \subset \mathbb C$
 is subharmonic
if for all $z_0\in G$ there exists $\varepsilon > 0$ such that
\begin{equation}\label{eq-1}
u(z_0)\le
\frac 1{2\pi}
\int_0^{2\pi} u(z_0+r e^{it})\, dt, \qquad 0<r<\varepsilon,
\end{equation}
If $u(z_0)$ = $|f(z_0)|^2 = 0,$ then \eqref{eq-1} holds. If $u(z_0) > 0,$
 then there exists a neighborhood
$U$ of $z_0$ such that $u$ is of class $C^2(U)$ (because the zeroes of $u$ are isolated),
and then we may prove that $\Delta u \ge 0$ on $U$. Thus the proof reduces to proving that
$\Delta u(z)\ge 0$ whenever $u(z)>0.$ In order to do this we will calculate $\Delta u.$

It is easy to prove that If $u > 0$ is a $C^2$ function defined on a region in $\mathbb C,$
 and $\alpha\in \mathbb R,$ then next two statements holds
\begin{equation}
\label{eq2}
\Delta(u^\alpha) = \alpha u^{\alpha-1}\Delta u+\alpha(\alpha-1)u^{\alpha-2}|\nabla u|^2,
\end{equation}
\begin{equation}
\label{eq-3}
|\nabla u|^2=4|\pa u|^2\quad\text{and}\quad \Delta u =4\pa\bar\pa u.
\end{equation}

\begin{lemma}\label{lem3}
If $f=g+\bar h,$ where $g$ and $h$ are holomorphic functions, then
\begin{equation}\label{eq-4}
    \Delta(|f|^2)=4(|g'|^2+|h'|^2).
\end{equation}
\end{lemma}
\begin{proof}
Since $|f|^2=(g+\bar h)(\bar g +h),$ we have
$$ \begin{aligned} \Delta (|f|^2)&=
4\pa(\overline{h'\mathstrut}(\bar g+ h)+(g+\bar h)\overline{ g'\mathstrut} )\\&
=4(\overline{ h'\mathstrut}h+ g\overline{ g'\mathstrut})\\&
=4(|g'|^2+|h'|^2).
\end{aligned}$$
\end{proof}

\begin{lemma}\label{lem4}
If $f=g+\bar h,$ where $g$ and $h$ are holomorphic functions, then
\begin{equation}\label{eq-5}
    |\nabla(|f|^2)|^2= 4(|\overline{ g'\mathstrut}|^2+|\overline{ h'\mathstrut}|^2)|f|^2
    + 8 \,{\rm Re}(\overline{ g'\mathstrut}h'f^2).
    \end{equation}
\end{lemma}

\begin{proof}
We have
$$
\begin{aligned}
|\nabla(|f|^2)|^2 &= 4 |\pa(|f|^2)|^2 \\&
= 4|\pa((g+\bar h)(\bar g+h))  |^2\\&
=4|g'\bar f + fh'|^2\\&
=4 (|g'|^2+|h'|^2)|f|^2 +  8 \,{\rm Re}(\overline{ g'\mathstrut}h'f^2).
\end{aligned}
$$
\end{proof}

\begin{lemma}\label{lem5}
If $f=g+\bar h,$ where $g$ and $h$ are holomorphic functions, then
\begin{equation}\label{eq-6}
    \Delta(|f|^p)=p^2(|g'|^2+|h'|^2)|f|^{p-2}+ 2p(p-2)|f|^{p-4}\,{\rm Re}(\overline{ g'\mathstrut}h'f^2)
\end{equation}
whenever $f\neq 0.$
\end{lemma}

\begin{proof}
We take $\alpha=p/2,$ $u=|f|^2,$ and then use \eqref{eq2},
\eqref{eq-4} and \eqref{eq-5} to get the result.
\end{proof}

\begin{subsec}{Proof of Theorem \ref{th-1}.} {\rm We have to prove that $\Delta
(|f|^p)\ge 0,$ where $p=4k/(1+k)^2.$ Since $p-2<0,$ we get from
\eqref{eq-6} that
$$
\begin{aligned}
\Delta(|f|^p) &\ge
p^2(|g'|^2+|h'|^2)|f|^{p-2}+2p(p-2)|f|^{p-4}|g'|\cdot |h'|\cdot
|f|^2\\& = p^2|g'|^2(m^2+1)|f|^{p-2}+2p(p-2)|g'|^2|f|^{p-2}m\\&
=p|g'|^2|f|^{p-2}[p(1+m^2)+2(p-2)m],
\end{aligned}
$$
where $m=|h'|/|g'|\le k.$ The function $m\mapsto p(1+m^2)+2(p-2)m$ has a negative
derivative (because $p<1$ and $m<1$), which implies that
$$ (1+m^2)p+2(p-2)m\ge (1+k^2)p+2(p-2)k.$$
On the other hand $(1+k^2)p+2(p-2)k\ge 0$ if and only if $p\ge 4k/(1+k)^2,$ which proves that $|f|^q$ is subharmonic. To prove that the exponent $q$ is optimal we take $f(z)=z+k\bar z.$ By \eqref{eq-6},
$$ \Delta (|f|^p)(1)= p^2(1+k^2)(1+k)^{p-2} +2p(p-2)(1+k)^{p-2}k.$$
Hence $\Delta (|f|^p)(1)\ge 0$ if and only if
$$ p(1+k^2)+2(p-2)k\ge 0,$$
which, as noted above, is equivalent to $p\ge q.$ This completes the
proof of Theorem ~ \ref{th-1}. } $\square$
\end{subsec}
\section{Moduli of continuity in Euclidean metric}

For a continuous  function $f:\overline{\mathbb D\mathstrut}\mapsto \mathbb C$ harmonic in
$\mathbb D$
 we define
two moduli of continuity:
$$  \omega(f,\delta)=\sup\{|f(e^{i\theta})-f(e^{it})|: |e^{i\theta}- e^{it}|\le \delta, \ t,\theta
\in\mathbb R\}, \quad \delta\ge 0,$$
and
$$ \tilde\omega(f,\delta)=\sup\{|f(z)-f(w)|: |z-w|\le \delta,\ z,w\in \overline{\mathbb
D\mathstrut}\}, \quad \delta\ge 0.$$
Clearly $\omega(f,\delta)\le \tilde \omega(f,\delta),$ but the reverse inequality need not hold.
To see this consider the function
$$
f(re^{i\theta})=\sum_{n=1}^\infty \frac {(-1)^n r^n\cos n\theta}{n^2}, \qquad re^{i\theta}\in \overline{\mathbb D\mathstrut}.
$$
This function is harmonic in $\mathbb D$ and continuous on $ \overline{\mathbb D\mathstrut}.$ The function  $v(\theta)=f(e^{i\theta})$, $ |\theta|<\pi,$ is differentiable, and
$$\begin{aligned}
\frac{dv}{d\theta}&=\sum_{n=1}^\infty \frac{(-1)^{n-1}\sin n\theta}{n}\\&
= \frac{\theta}2,
\qquad |\theta|<\pi.
\end{aligned}
$$
This formula is well known, and can be verified by calculating the Fourier coefficients
of the function $\theta\mapsto \theta/2,\ |\theta|<\pi.$ It follows that
$$ |f(e^{i\theta})-f(e^{it})|\le (\pi/2)|\theta-t|,\qquad -\pi<\theta,\,t<\pi,$$
and hence
$ \omega(f,\delta)\le M \delta, \ \delta>0,$
where $M$ is an absolute constant.
On the other hand, the inequality $\tilde \omega(f,\delta)\le CM\delta$, $C={\rm const},$ does not hold
because it  implies that $|\partial f/\partial r|\le CM,$ which is not true because
$$ \frac {\pa}{\pa r}f(re^{i\theta})=\sum_{n=1}^\infty \frac{r^{n-1}}n, \quad\text{for $\theta=\pi,\ 0<r<1$.}$$

However, as was proved by Rubel, Shields and Taylor \cite{rst}, and Tamrazov \cite{ta}, if $f$ is a holomorphic function,
then $ \tilde \omega (f,\delta)\le C  \omega (f,\delta)$, where $C$ is independent of $f$
and $\delta.$ Here we extend that result to quasiregular harmonic functions.

\begin{theorem}\cite{kp}\label{th-2}
Let $f$ be a  $k$-quasiregular harmonic complex-valued function which has a continuous extension on
$\overline{\mathbb D\mathstrut}$, then there is a constant $C$ depending only on $k$ such that
$ \tilde \omega (f,\delta)\le C  \omega (f,\delta)$.
\end{theorem}

In order to deduce this fact from Theorem \ref{th-1}, we need some simple properties of
the modulus $\omega(f,\delta).$ Let
 $$\omega_0(f,\delta)=\sup\{|f(e^{i\theta})-f(e^{it})|: |\theta- t|\le \delta,
\ t,\theta
\in\mathbb R\}.$$
It is easy to check that \begin{equation}\label{eq-0}
C^{-1}\omega_0(f,\delta)\le \omega(f,\delta)\le C\omega_0(f,\delta),
\end{equation}
where $C$ is an absolute constant, and that
\begin{equation*}
    \omega_0(f,\delta_1+\delta_2)\le \omega_0(f,\delta_1)+\omega_0(f,\delta_2),
    \quad \delta_1,\,\delta_2\ge 0.
\end{equation*}
Hence, $\omega_0(f,2^n\delta)\le 2^n\omega_0(f,\delta),$ and hence
$
\omega_0(\lambda\delta)\le 2\lambda\omega_0(\delta),$ for  $\lambda\ge 1,\delta\ge 0.
$
From these inequalities  and \eqref{eq-0} it follows that
\begin{equation}\label{eq-x}
    \omega (f,\lambda\delta)\le 2C\lambda\omega(f,\delta),\quad \lambda\ge 1, \delta\ge 0,
\end{equation}
and
\begin{equation}\label{eq-sub}
    \omega(f,\delta_1+\delta_2)\le C\omega(f,\delta_1)+C\omega(f,\delta_2),
    \quad \delta_1,\,\delta_2\ge 0,
\end{equation}
where $C$ is an absolute constant.
As a consequence of \eqref{eq-x} we have, for $0<p<1,$
\begin{equation}\label{eq-2}
    \int_x^\infty \frac{\omega(f,t)^p}{t^2}\,dt \le C \frac{\omega(f,x)^p}x, \quad x>0,
\end{equation}
where $C$ depends only on $p.$ Finally we need the following
consequence of the harmonic Schwarz lemma (see \cite{abr}).

\begin{lemma}
If $h$ is a function harmonic and bounded in the unit disk, with $h(0)=0,$
the $|h(\xi)|\le (4/\pi)\|h\|_\infty|\xi|,$ for $\xi\in \mathbb D.$
\end{lemma}

\begin{subsec}{Proof of Theorem \ref{th-2}.} {\rm It is enough to prove
that $|f(z)-f(w)|\le C\omega(f,|z-w|)$ for all $z,\,w\in
\overline{\mathbb D\mathstrut},$ where $C$ depends only on $k.$
Assume first that $z=r\in (0,1)$ and  $|w|=1.$ Then, by Theorem
\ref{th-1}, the function $\varphi(\xi)=|f(w)-f(\xi)|^q ,$ where
$q=4k/(1+k)^2<1,$ is subharmonic in $\mathbb D$ and continuous on
$\overline{\mathbb D\mathstrut},$ whence
$$ \begin{aligned}
\varphi(r)\le \frac 1{2\pi}\int_{\partial\mathbb D} \frac {(1-r^2)
\varphi(\zeta)}{|\zeta-r|^2}\,|d\zeta|.
\end{aligned}
$$
Since, by \eqref{eq-sub},
$$
\begin{aligned}
\varphi(\zeta)&\le (\omega(f,|w-r|+|r-\zeta|))^q
\\&\le C^q\omega(f,|w-r|)^q+C^q\omega(f,|r-\zeta|)^q,
\end{aligned}
$$
we have
$$\begin{aligned}
 \varphi(z)&\le C^q\omega(f,|w-r|)^q+
\frac {C^q}{2\pi}\int_{\partial\mathbb D} \frac {(1-r^2)\omega(f,|r-\zeta|)^q}{|\zeta-r|^2}\,|d\zeta|\\&
=C^q\omega(f,|w-r|)^q+ \frac{C^q}{2\pi}\int_{-\pi}^\pi \frac {(1-r^2)
\,\omega(|r-e^{it}|)^q}{|e^{it}-r|^2}\,dt.
\end{aligned}$$
But simple calculation shows that
$$
|r-e^{it}|=\sqrt{(1-r)^2+4r\sin^2(t/2)\mathstrut}\asymp 1-r+|t|\quad (0<r<1,\ |t|\le\pi).
$$
From this, \eqref{eq-1}, and \eqref{eq-2} it follows that
$$
\begin{aligned}
\int_{-\pi}^\pi \frac {(1-r^2)\,\omega(f,|r-e^{it}|)^q}{|e^{it}-r|^2}\,dt&\le
C_1\int_0^\pi \frac{(1-r)\,\omega(f,1-r+t)^q}{(1-r+t)^2}\,dt
\\& =C_1\Big(\int_0^{1-r}+ \int_{1-r}^\pi\Big)\frac{(1-r)\,\omega(f,1-r+t)^q}{(1-r+t)^2}\,dt
\\& \le C_2\,(\omega(1-r))^q+ C_2\,(1-r)\int_{1-r}^\infty  \frac{\omega(f,t)^q}{t^2}\,dt\\&
\le C_3\,(\omega(f,1-r))^q\\&
\le C_4\,(\omega(f,|w-z|))^q.
\end{aligned}
$$
Thus $|f(w)-f(z)|\le C_5\omega(f,|w-z|)$ provided $w\in\partial\mathbb D$ and $z\in (0,1).$
By rotation and the continuity of $f$, we can extend this inequality to the case where $w\in\partial\mathbb D$ and
$z\in \overline{\mathbb D\mathstrut}.$

 If $0<|w|<1,$
we consider the function $h(\xi)=f(\xi w/|w|)-f(\xi z/|w|)$, $|\xi|\le 1.$
This function is harmonic in $\mathbb D,$ continuous on $\overline{\mathbb D\mathstrut},$  and $h(0)=0.$
Hence, by the harmonic Schwarz lemma, inequality \eqref{eq-1}, and the preceding case,
$$\begin{aligned}
|f(w)-f(z)|& = |h(|w|)|\\&\le (4/\pi)|w|\,\|h\|_\infty\\&
\le C_6|w|\,\omega(f,|w/|w|- z/|w|\,|)\\&
\le C_7 \,\omega(f,|w|\,|\,|w/|w|- z/|w|\,|)\\&
=C_7\,\omega(f,|w- z|),
\end{aligned}
$$
which completes the proof. $\square$ }
\end{subsec}

\section{Lipschitz continuity up to the boundary on $B^n$}

It is known, even for $n=2$, that Lipschitz continuity of $\phi : T \rightarrow C$,
where $T = \{ z \in C : |z| = 1 \}$, does not imply Lipschitz continuity of $u = P[ \phi ]$.

Here, for any $n \geq 2$,

$$P[ \phi ] (x) = \int_{S^{n-1}} P(x, \xi) \phi (\xi) d\sigma (\xi), \;\; x \in B^n$$
where $P(x, \xi) = \frac{1 - |x|^2}{|x - \xi |^n}$ is the Poisson kernel for the unit ball
$B^n = \{ x \in R^n : |x| < 1 \}$, $d \sigma$ is the normalized surface measure on the unit sphere
$S^{n-1}$ and $\phi : S^{n-1} \rightarrow R^n$ is a continuous mapping.

Our aim is to show that Lipschitz continuity is preserved by harmonic extension, if
the extension is quasiregular. The analogous statement is true for H\"older continuity
without assumption of quasiregularity.


\begin{theorem}
\cite{akm}
Assume $\phi : S^{n-1} \rightarrow R^n$ satisfies a Lipschitz condition:
$$ | \phi(\xi) - \phi(\eta) | \leq L | \xi - \eta |, \;\; \xi, \eta \in S^{n-1}$$
and assume $ u = P[ \phi ] : B^n \rightarrow R^n$ is $K$-quasiregular. Then
$$ | u(x) - u(y) | \leq C^\prime |x-y|, \;\; x, y \in B^n$$
where $C^\prime$ depends on $L$, $K$ and $n$ only.

\end{theorem}

D. Kalaj obtained a related result, but under additonal assumption of $C^{1, \alpha}$
regularity of $\phi$, (see \cite{ka}).
\begin{proof}
The main part of the proof is the estimate of the tangential derivatives of $u$, and in that
 part quasiregularity plays no role. We choose $x_0 = r \xi_0 \in B^n$,
$r = |x|$, $\xi_0 \in S^{n-1}$. Let $T = T_{x_0} rS^{n-1}$ be the
$n-1$ dimensional tangent plane at $x_0$ to the sphere $rS^{n-1}$. We want to prove that

\begin{equation}
\label{tander}
\Vert D(u|_T)(x_0) \Vert \leq C(n) L.
\end{equation}
Without loss of generality we can assume  $\xi_0 = e_n$ and $x_0 = re_n$. By a simple
calculation

$$
\frac{\partial}{\partial x_j} P(x, \xi) = \frac{-2x_j}{|x-\xi |^n} - n(1 - |x|^2)
\frac{x_j - \xi_j}{|x-\xi |^{n+2}}.
$$

Hence, for $1 \leq j < n$ we have

$$\frac{\partial}{\partial x_j} P(x_0, \xi) = n(1 - |x_0|^2) \frac{\xi_j}{|x_0 - \xi |^{n+2}}.$$

It is important to note that this kernel is odd in $\xi$ (with respect to reflection $(\xi_1, \ldots \xi_j, \ldots ,\xi_n)
\mapsto (\xi_1, \ldots , -\xi_j, \ldots, \xi_n)$), a typical fact for kernels obtained by differentiation.
This observation and differentiation under integral sign gives, for any $1 \leq j < n$,

\begin{eqnarray*}
\frac{\partial u}{\partial x_j} (x_0) & = & n (1 - r^2)
\int_{S^{n-1}} \frac{\xi_j}{|x_0 - \xi |^{n+2}} \phi( \xi) d\sigma(\xi) \cr
& = & n (1 - r^2)
\int_{S^{n-1}} \frac{\xi_j}{|x_0 - \xi |^{n+2}} ( \phi( \xi) - \phi(\xi_0)) d\sigma(\xi). \cr
\end{eqnarray*}

Using the elementary inequality $|\xi_j| \leq | \xi - \xi_0 |$, ($ 1 \leq j < n$, $\xi \in S^{n-1}$)
and Lipschitz continuity of $\phi$ we get

\begin{eqnarray*}
\left| {\partial u \over \partial x_j}(x_0) \right| & \leq  &
Ln(1- r^2) \int_{S^{n-1}} {  |\xi_j| |\xi - \xi_0 | \over |x_0 - \xi|^{n+2} } d\sigma(\xi) \cr
& \leq & Ln(1- r^2)
\int_{S^{n-1}} { |\xi - \xi_0 |^2 \over |x_0 - \xi|^{n+2} } d\sigma(\xi). \cr
\end{eqnarray*}

In order to estimate the last integral, we split $S^{n-1}$ into two subsets
$E = \{ \xi \in S^{n-1} : |\xi - \xi_0 | \leq 1-r \}$ and
$F = \{ \xi \in S^{n-1} : |\xi - \xi_0 | > 1-r \}$. Since
$| \xi - x_0 | \geq 1 - |x_0|$ for all $\xi \in S^{n-1}$ we have

\begin{eqnarray*}
\int_E { | \xi - \xi_0 |^2 \over |x_0 - \xi |^{n+2} } d \sigma(\xi) & \leq & (1-r^2)^{-n-2}
\int_E |\xi - \xi_0|^2 d \sigma (\xi) \cr
& \leq & (1-r^2)^{-n-2} \int_0^{1-r} \rho^2 \rho^{n-2} d \rho \cr
& \leq & {2 \over n+1} (1- r)^{-1}.\cr
\end{eqnarray*}

On the other hand, $ | \xi - \xi_0 | \leq C_n |\xi - x_0 |$ for every $\xi \in F$, so

\begin{eqnarray*}
\int_F { | \xi - \xi_0 |^2 \over |x_0 - \xi |^{n+2} } d \sigma(\xi) & \leq & C_n^{n+2}
\int_F |\xi - \xi_0|^{- n} d\sigma(\xi) \cr
& \leq & C_n^\prime \int_{1-r}^2 \rho^{-n} \rho^{n-2} d \rho \cr
& \leq & C_n^\prime (1-r)^{-1}. \cr
\end{eqnarray*}

Combining these two estimates we get
$$ \left| {\partial u \over \partial x_j } (x_0) \right| \leq L C(n)$$
for $1 \leq j < n$. Due to rotational symmetry, the same estimate holds for every derivative
in any tangential direction. This establishes estimate (\ref{tander}).  Finally,
$K$-quasiregularity gives

$$ \Vert Du(x) \Vert \leq LKC(n).$$

Now the mean value theorem gives Lipschitz continuity of $u$.
\end{proof}

\begin{problem}\rm
\begin{enumerate}
\item
Can one prove similar result for other type of moduli of continuity, as was done
in Section 2 in the planar case?
\item
The same questions can be posed in other smoothly bounded domains.
\end{enumerate}
\end{problem}

\section{Bilipschitz maps}

Bilipschitz property of harmonic quasiconformal mappings on the unit
disc was investigated in \cite{mat}. A different approach to the
following theorem is given in \cite{mat1}.

\begin{theorem}
\label{bilipschitz}
Suppose $D$ and $D'$ are proper domains in $\mathbb  R^2$. If $f:D\longrightarrow D'$ is $K$-qc and harmonic,
then it is bilipschitz with respect to quasihyperbolic metrics on $D$ and $D'$.
\end{theorem}

\begin{proof}
Since $f$ is harmonic we have locally, representation
$$
f(z)=g(z)+\overline{h(z)},
$$
where $g$ and $h$ are analytic functions. Then Jacobian $J_f(z)=|g'(z)|^2-|h'(z)|^2>0$
(note that $g'(z)\neq 0$).

Futher,
$$
J_f(z)=|g'(z)|^2\left(1-\frac{|h'(z)|^2}{|g'(z)|^2}\right)=|g'(z)|^2\left(1-|\omega(z)|^2\right),
$$
where $\omega(z)=\frac{h'(z)}{g'(z)}$ is analytic and $|\omega|<1$. Now we have
$$
\log\frac{1}{J_f(z)}=-2\log|g'(z)|-\log(1-|\omega(z)|^2).
$$
The first term is harmonic function (it is well known that logarithm of moduli of analytic function
is harmonic everywhere except where that analytic function vanishes, but $g'(z)\neq 0$ everywhere).

The second term can be expanded in series
$$
\sum_{k=1}^{\infty}\frac{|\omega(z)|^{2k}}{k},
$$
and each term is subharmonic (note that $\omega$ is analytic).

So, $-\log(1-|\omega(z)|^2)$ is a continuous function represented as a locally uniform
sum of subharmonic functions. Thus it is also subharmonic.

Hence
\begin{equation}
\label{first}
\log\frac{1}{J_f(z)} \mbox{ is a subharmonic function}.
\end{equation}

Note that representation $f(z)=g(z)+\overline{h(z)}$ is local, but that suffices for our conclusion
(\ref{first}).

By the definition from [AG, Definition 1.5]
$$
\alpha_f(z)=\exp\left(\frac{1}{n}(\log J_f)_{B_z}\right),
$$
where
$$
(\log J_f)_{B_z}=\frac{1}{m(B_z)}\int_{B_z}\log J_f\,dm, \quad B_z=B(z,d(z,\partial D)).
$$
In the case $n=2$ we have
\begin{equation}
\label{second}
\frac{1}{\alpha_f(z)}=\exp\left(\frac 12\frac 1{m(B_z)}\int_{B_z}\log\frac{1}{J_f(w)}\,dm(w)\right).
\end{equation}
From (\ref{first}) we have
$$
\frac 1{m(B_z)}\int_{B_z}\log\frac{1}{J_f(w)}\,dm(w)\geq\log\frac{1}{J_f(z)}.
$$
Combining this with (\ref{second}) we have
$$
\frac{1}{\alpha_f(z)}\geq\exp\left(\frac 12\log\frac{1}{J_f(z)}\right)=\frac{1}{\sqrt{J_f(z)}}
$$
and therefore
$$
\sqrt{J_f(z)}\geqslant\alpha_f(z).
$$
On the other hand, we have Theorem [AG, Theorem 1.8]:

Suppose that $D$ and $D'$ are domains in $\mathbb R^n$ if $f:D\longrightarrow D'$ is $K$-qc,
then
$$
\frac 1c\frac{d(f(x),\partial D')}{d(x,\partial D)}\leq
\alpha_f(z)\leq
c\,\frac{d(f(x),\partial D')}{d(x,\partial D)}
$$
for $x\in D$, where $c$ is a constant wich depends only on $K$ and $n$.

From first inequality of this theorem we have
\begin{equation}
\label{jacobian_root}
\sqrt{J_f(z)}\geq\frac 1c\frac{d(f(x),\partial D')}{d(x,\partial D)}.
\end{equation}

Note that
$$
J_f(z)=|g'(z)|^2-|h'(z)|^2\leq|g'(z)|^2
$$
and by $K$-qclity of $f$, $|h'|\leq k|g'|$, $0\leq k<1$, where $K=\frac{1+k}{1-k}$.

This gives $J_f\geq (1-k^2)|g'|^2$. Hence,
$$
\sqrt{J_f}\asymp |g'|\asymp |g'|+|h'|=||f'(z)||.
$$
Finally (\ref{jacobian_root}) and the above asymptotic relation give
$$
||f'(z)||\geq\frac 1c\frac{d(f(x),\partial D')}{d(x,\partial D)},\quad c=c(k).
$$
For the reversed inequality we again use $J_f(z)\geq(1-k^2)|g'(z)|^2$, i.e.
\begin{equation}
\label{four}
\sqrt{J_f(z)}\geq\sqrt{1-k^2}|g'(z)|
\end{equation}
Further, we know that for $n=2$
$$
\alpha_f(z)=\exp\left(\frac{1}{m(B_z)}\int_{B_z}\log\sqrt{J_f(x)}\,dm(w)\right).
$$
Using (\ref{four})
$$
\begin{array}{rcl}
\displaystyle
\frac{1}{m(B_z)}\int_{B_z}\log\sqrt{J_f(x)}\,dm(w) & \geq &
\displaystyle \frac{1}{m(B_z)}\int_{B_z}\log\sqrt{1-k^2}+\log|g'(w)|\,dm(w)
\vspace{0.5em}\\
& = &
\displaystyle
\log\sqrt{1-k^2}+\frac{1}{m(B_z)}\int_{B_z}\log|g'(w)|\,dm(w)
\vspace{0.5em}\\
& = & \log\sqrt{1-k^2}+\log|g'(z)|.
\end{array}
$$
Now we have by harmonicity of $\log|g'|$
$$
\begin{array}{rcl}
\displaystyle
\alpha_f(z) & = & \displaystyle \exp\left(\frac{1}{m(B_z)}\int_{B_z}\log\sqrt{J_f(x)}\,dm(w)\right)
\vspace{0.5em}\\
& \geq &
\displaystyle
\exp(\log\sqrt{1-k^2}+\log|g'(z)|)
\vspace{0.5em}\\
& = &
\displaystyle
\sqrt{1-k^2}|g'(z)|
\vspace{0.5em}\\
& \geq &
\displaystyle
\frac 12\sqrt{1-k^2}(|g'|+|h'|)
\vspace{0.5em}\\
& = &
\displaystyle
\frac{\sqrt{1-k^2}}{2}||f'||.
\end{array}
$$
Again using the second inequality in [AG, Theorem 1.8]
$$
||f'||\leqslant c\sqrt{J_f(z)}\leqslant c\,\alpha_f(z)\leqslant c\,\frac{d(f(z),\partial D')}{d(z,\partial D)},
\quad c=c(k).
$$
Summarizing
$$
||f'(z)||\asymp\frac{d(f(z),\partial D')}{d(z,\partial D)}.
$$
This pointwise result, via integration along curves, easily gives
$$
k_{D'}(f(z_1),f(z_2))\asymp k_D(z_1,z_2).
$$
\end{proof}
\rule{0mm}{5mm}

\begin{problem}\rm
Is Theorem \ref{bilipschitz} true in dimensions $n\ge 3$?
\end{problem}
%
%

\end{document}